\documentclass[11pt]{amsart}
\usepackage{tikz,amsthm,amsmath,amstext,amssymb,amscd,epsfig,euscript, mathrsfs, dsfont,pspicture,multicol,graphpap,graphics,graphicx,times,enumerate,subfig,sidecap,wrapfig,color}
\usepackage{amsmath}
\usepackage{amssymb}
\usepackage{pdfsync}

\newcommand{\hm}{{\omega}}

\newcommand{\om}{{\Omega}}

\newcommand{\R}{{\mathbb R}}       
\newcommand{\Z}{{\mathbb Z}}       
\newcommand{\DD}{{\mathcal D}}
\newcommand{\CC}{{\mathcal C}}

\newcommand{\HH}{{\mathcal H}}

\newcommand{\UU}{{\mathcal U}}

\newcommand{\RR}{{\mathcal R}}
\newcommand{\NN}{{\mathcal N}}
\newcommand{\MD}{{\mathcal{MD}}}

\newcommand{\EE}{{\mathcal E}}

\newcommand{\diam}{\mathop{\rm diam}}
\newcommand{\dist}{{\rm dist}}

\newcommand{\fiproof}{{\hspace*{\fill} $\square$ \vspace{2pt}}}

\newcommand{\rf}[1]{{(\ref{#1})}}

\newcommand{\supp}{\operatorname{supp}}

\newcommand{\vphi}{{\varphi}}
\newcommand{\ve}{{\varepsilon}}
\newcommand{\vv}{{\vspace{2mm}}}
\newcommand{\vvv}{\vspace{4mm}}
\newcommand{\wt}[1]{{\widetilde{#1}}}
\newcommand{\wh}[1]{{\widehat{#1}}}

\newcommand{\pom}{{\partial \Omega}}

\newcommand{\bad}{{\mathsf{Bad}}}
\newcommand{\good}{{\mathsf{Good}}}

\newcommand{\sss}{{\mathsf{Stop}}}
\newcommand{\ttt}{{\mathsf{Top}}}

\newcommand{\nex}{{\mathsf{Next}}}

\def\Xint#1{\mathchoice
{\XXint\displaystyle\textstyle{#1}}%
{\XXint\textstyle\scriptstyle{#1}}%
{\XXint\scriptstyle\scriptscriptstyle{#1}}%
{\XXint\scriptscriptstyle\scriptscriptstyle{#1}}%
\!\int}
\def\XXint#1#2#3{{\setbox0=\hbox{$#1{#2#3}{\int}$ }
\vcenter{\hbox{$#2#3$ }}\kern-.58\wd0}}

\def\avint{\Xint-}

\newtheorem{theorem}{Theorem}[section]
\newtheorem{lemma}[theorem]{Lemma}

\newtheorem{mlemma}[theorem]{Main Lemma}
\newtheorem{keylemma}[theorem]{Key Lemma}
\newtheorem{coro}[theorem]{Corollary}

\newtheorem{claim}{Claim}

\newtheorem*{lemma*}{Lemma}
\newtheorem*{theorem*}{Theorem}
\newtheorem*{theoremA}{Theorem A}
\newtheorem*{theoremB}{Theorem B}

\theoremstyle{definition}

\theoremstyle{remark}
\newtheorem{rem}[theorem]{\bf Remark}

\numberwithin{equation}{section}

\newcommand{\brem}{\begin{rem}}
\newcommand{\erem}{\end{rem}}

\def\d{\partial}

\begin{document}

\begin{abstract}
Let $\Omega\subsetneq\R^{n+1}$ be open and let $\mu$ be some measure supported on $\partial\Omega$ such that
 $\mu(B(x,r))\leq C\,r^n$ for all $x\in\R^{n+1}$, $r>0$.
We show that if the harmonic measure in $\Omega$ satisfies 
some scale invariant $A_\infty$ type conditions with respect to $\mu$, then the $n$-dimensional Riesz transform
$$\RR_\mu f(x) = \int \frac{x-y}{|x-y|^{n+1}}\,f(y)\,d\mu(y)$$
is bounded in $L^2(\mu)$. We do not assume any doubling condition on $\mu$. We also consider the particular case
when $\Omega$ is a bounded uniform domain. To this end, we need first to obtain sharp estimates that
relate the
harmonic measure and the Green function in this type of domains, which generalize classical results by Jerison and Kenig for the
well-known class of NTA domains. 
\end {abstract}

\title{Harmonic measure and Riesz transform in uniform and general domains}

\author{Mihalis Mourgoglou}
\address{Departament de Matem\`atiques\\ Universitat Aut\`onoma de Barcelona and Centre de Reserca Matem\` atica\\ Edifici C Facultat de Ci\`encies\\
08193 Bellaterra (Barcelona) }
\email{mmourgoglou@crm.cat}
\author{Xavier Tolsa}
\address{ICREA, Pg. Lluís Companys 23, 08010 Barcelona, Catalonia, and
Dept. of Mathematics and BGSMath, Universitat Autònoma de Barcelona, 08193 Bellaterra, Catalonia}
\email{xtolsa@mat.uab.cat}
\thanks{The authors were supported by the ERC grant 320501 of the European Research Council (FP7/2007-2013). X.T. was also supported by 2014-SGR-75 (Catalonia), MTM2013-44304-P (Spain), and by the Marie Curie ITN MAnET (FP7-607647).}

\maketitle

\section{Introduction}

In this paper we study the relationship between harmonic measure in a general domain $\Omega\subset\R^{n+1}$  and the $L^2$ boundedness
of the $n$-dimensional Riesz transform with respect to some measure $\mu$ supported on $\partial\Omega$. We do not assume any doubling condition on the surface measure of $\partial\Omega$ or on the underlying measure $\mu$. 
 We also consider the particular case when the domain $\Omega$ is a uniform domain.
Further, for this type of domains we obtain
sharp estimates which relate the harmonic measure and the Green function on $\Omega$ which are of independent interest and are new in such generality, as far as we know.

Let $n\geq1$, let $\Omega\subsetneq \R^{n+1}$ be an open set, and let  
$\mu$  be a Radon
measure supported on $\partial \Omega$ satisfying the growth condition
\begin{equation}\label{eqgrowth**}
\mu(B(x,r))\leq C_\mu\,r^n\qquad \mbox{for all $x\in\R^{n+1}$ and all $r>0$.}
\end{equation}
Roughly speaking, our first theorem asserts that if the harmonic measure in $\Omega$ satisfies some scale invariant $A_\infty$ type condition with respect to $\mu$, then the Riesz transform 
$$\RR_\mu f(x) = \int \frac{x-y}{|x-y|^{n+1}}\,f(y)\,d\mu(y)$$
is bounded in $L^2(\mu)$.
To state the theorem in detail, we need some additional notation and terminology.

Given a point $p\in\Omega$, we denote by $\omega^p$ the harmonic measure in $\Omega$ with pole $p$.
Given $a,b>1$,
we say that a ball $B\subset \R^{n+1}$ is $\mu$-$(a,b)$-doubling for $\mu$ (or just $(a,b)$-doubling if the 
measure $\mu$ is clear from the context)
if 
$$\mu(aB)\leq b\,\mu(B),$$ 
where $aB$ stands for the ball concentric with $B$ with radius $a$ times the radius of $B$.

Our main result is the following:

\begin{theorem}\label{teo1}
Given $n\geq 1$, let $0<\kappa<1$ be some constant small enough and
$c_{db}>1$ another constant big enough, both depending only on $n$.
Let $\Omega$ be an open set in $\R^{n+1}$ and $\mu$ be a Radon measure supported on $\d \om$ 
satisfying the growth condition \rf{eqgrowth**}.
 Suppose that 
 there exist $\ve, \ve' \in (0,1)$ such that for every $\mu$-$(2,c_{db})$-doubling ball $B$ centered at $\supp\mu$ with $\diam(B)
 \leq\diam(\supp\mu)$
 there exists a point
 $x_B\in \kappa B\cap\Omega$ such that the
 following holds: for any subset $E \subset B$,
 \begin{equation}\label{eq:Amu1}
\mbox{if}\quad\mu(E)\leq \ve\,\mu(B), \quad\text{ then }\quad \hm^{x_B}(E)\leq \ve'\,\hm^{x_B}(B).
\end{equation}
Then the Riesz transform $\RR_\mu: L^2(\mu) \to L^2(\mu)$ is bounded.
\end{theorem}
\vv

Let us remark that it does not matter if in the theorem the balls $B$ are assumed to be either open of closed.
Observe that we do not ask the pole $x_B$ to be at some distance from $\partial\Omega$ comparable to $\diam(B)$. On the contrary,
$x_B$ can be arbitrarily close to $\partial\Omega$.
Notice also that, by taking complements, we deduce that if $\mu$ and $\hm^{x_B}$ satisfy the conditions above
for a fixed $(2,c_{db})$-doubling ball $B$ centered at $\supp\mu$, then the following holds: for any subset $E\subset B$,
$$\text{if }\quad \hm^{x_B}(E)<(1-\ve')\,\hm^{x_B}(B),\quad\text{ then }\quad
\mu(E)<(1-\ve)\,\mu(B).$$

Under the assumptions of the theorem, in the particular case when $\mu$ is mutually absolutely continuous with respect to the Hausdorff measure $\HH^n$ on a subset $E\subset
\partial\Omega$,
 we deduce that $E$ is $n$-rectifiable, by the Nazarov-Tolsa-Volberg theorem \cite{NToV-pubmat}. Further, when 
 $\mu=\HH^n|_E$ and $E$ is AD-regular, we infer that $E$ is uniformly rectifiable, by \cite{NToV}, and we ``essentially'' reprove
 (by different methods)
 a recent result of Hofmann and Martell \cite{HM4}. See the next section for the notions of AD-regularity and uniform rectifiability.
Our theorem extends to a more general framework some of the recent results in \cite{HM4}, where the AD-regularity of the surface measure $\HH^n|_{\partial\Omega}$ is a basic assumption. See Section \ref{sec11} for more details about how 
Theorem \ref{teo1} specializes when $\mu$ is AD-regular and how this is connected to the main result in \cite{HM4}.
Let us also mention that, under the assumption that $\partial \Omega$ is AD-regular, an interesting partial converse in terms of ``big pieces'' to the aforementioned
result from \cite{HM4} has been obtained recently by Bortz and Hofmann in \cite{BH}.

When the measure $\mu$ is not absolutely continuous with respect to the Hausdorff measure $\HH^n$, then from the $L^2(\mu)$ boundedness of $\RR_\mu$ we cannot deduce
that $\mu$ is $n$-rectifiable. However, in this situation the $L^2$ 
boundedness of the Riesz transform still provides some geometric information on $\mu$. This is specially clear when 
$n=1$, as shown in the works \cite{Tolsa-bilip} and \cite{AT}, for example.

We also remark that 
Theorem \ref{teo1} can be considered as a local quantitative version of the main theorem in 
\cite{AHM3TV}, where it is shown that if the harmonic measure and the Hausdorff measure $\HH^n$ are mutually absolutely 
continuous in some subset $E\subset\partial \Omega$ with $0<\HH^n(E)<\infty$, then $E$ is $n$-rectifiable. To prove this, 
it is shown in \cite{AHM3TV} that any such set $E$ contains another subset $F\subset E$ with $\HH^n(F)>0$ such that
$\RR_{\HH^n|_F}$ is bounded in $L^2(\HH^n|_F)$. Some of the arguments to prove Theorem \ref{teo1} are inspired by the 
techniques in \cite{AHM3TV}.

\vv

In this paper we also consider the particular case when $\Omega$ is a bounded uniform domain in $\R^{n+1}$, that is, a 
bounded domain satisfying the interior corkscrew and the Harnack chain conditions (see the next section for the precise definitions). For this type
of domains
 a variant of the preceding theorem with the harmonic measure
with respect to a fix pole $p$ holds. Now the assumption \rf{eq:Amu1} is replaced by a weaker (apparently) variant of the well
known $A_\infty$ condition. Let $\mu$ and $\sigma$ be Radon measures in $\R^{n+1}$. 
For $c_{db}>1$ and $0<\ve,\ve'<1$,
we write
$\sigma\in \wt A_\infty(\mu,c_{db},\ve,\ve')$ if for every $\mu$-$(2,c_{db})$-doubling ball $B$ centered at $\supp\mu$ 
 with $\diam(B)
 \leq\diam(\supp\mu)$ the
  following holds: for any subset $E \subset B$,
 \begin{equation}\label{eq:Amu1**}
\mbox{if}\quad\mu(E)\leq \ve\,\mu(B), \quad\text{ then }\quad \sigma(E)\leq \ve'\,\sigma(B).
\end{equation}
It is easy to check that if $\sigma\in \wt A_\infty(\mu,c_{db},\ve,\ve')$, then $\mu$ and $\sigma$ are mutually absolutely continuous
on $\supp\mu$. The condition $\sigma\in \wt A_\infty(\mu,c_{db},\ve,\ve')$ can be considered as a quantitative version of this fact.

Then we have:

\begin{theorem}\label{teo2}
Let $n\geq 1$, $\Omega$ be a bounded uniform domain in $\R^{n+1}$ and $\mu$ be a Radon measure supported on $\d \om$ 
satisfying the growth condition \rf{eqgrowth**}. Let $c_{db}>1$ be some constant big enough depending only on $n$
and let $0<\ve,\ve'<1$. Let $p\in\Omega$ and suppose that $\omega^p\in\wt A_\infty(\mu,c_{db},\ve,\ve')$.
Then the Riesz transform $\RR_\mu: L^2(\mu) \to L^2(\mu)$ is bounded.
\end{theorem}

Analogously to Theorem \ref{teo1}, when $\mu$ coincides with $\HH^n|_{\partial\Omega}$ and is AD-regular,
 by \cite{NToV} it follows that $\partial \Omega$ is uniformly rectifiable (see Section \ref{secprelim} for the definition).   This corollary was  previously
obtained by Hofmann, Martell and Uriarte-Tuero \cite{HMU} by quite different arguments. Further, we remark that in this case the converse
statement is also true, by another theorem due to Hofmann and Martell \cite{HM1}. An alternative argument for this converse implication 
appears in the recent work \cite{AHMNT},
where it is shown
that any uniform domain with uniformly rectifiable boundary is an NTA domain and then, by a well-known result of David and Jerison
\cite{DJ},  $\omega^p$ is an $A_\infty(\HH^n|_{\partial\Omega})$
weight. 
 So notice that for a bounded uniform domain whose
boundary is AD-regular, the following nice characterization holds:$$\mbox{\em 
$\partial \Omega$ is uniformly $n$-rectifiable if and only if $\omega^p$ is an $A_\infty(\HH^n|_{\partial\Omega})$
weight.}$$

\vv
Theorem \ref{teo2} follows from Theorem \ref{teo1} and the following technical result, which may be of independent interest.

\begin{theorem}\label{teounif}
Let $n\geq 1$, $\Omega$ be a uniform domain in $\R^{n+1}$ and let $B$ be a ball centered at $\partial\Omega$.
Let $p_1,p_2\in\Omega$ such that $\dist(p_i,B\cap \partial\Omega)\geq c_0^{-1}\,r(B)$ for $i=1,2$.
Then, for any Borel set $E\subset B\cap\d\om$,
$$\frac{\omega^{p_1}(E)}{\omega^{p_1}(B)}\approx \frac{\omega^{p_2}(E)}{\omega^{p_2}(B)},$$
with the implicit constant depending only on $c_0$ and the uniform behavior of $\Omega$. 
\end{theorem}

This result is already known to hold for the class of NTA domains introduced by Jerison and Kenig \cite{JK} and also for the uniform domains
satisfying the capacity density condition of Aikawa \cite{Ai2}. However it seems to be new for the case
of arbitrary uniform domains. To prove Theorem \ref{teounif} we study first the relationship between
harmonic measure and Green's function in this type of domains. In particular, in the case $n\geq2$ we show that if 
$B$ is a ball with radius $r$ centered at $\partial\Omega$ and $x_B\in\Omega$ is a corkscrew point for $B$ (see Section \ref{secprelim} 
for the precise definition), then
$$\omega^{x}(B)\approx \omega^{x_{B}}(B)\, r^{n-1}\, G(x,x_{B})\,\, \,\,\,\,\,\,\text{for all}\,\, x\in \Omega\backslash 2B.
$$
If $\Omega$ is an NTA domain or a uniform domain satisfying the capacity density condition, then $\omega^{x_{B}}(B)\approx1$ and the preceding estimate reduces to well known results  due respectively to Jerison and Kenig \cite{JK} and to Aikawa \cite{Ai2}.

\vv
The plan of the paper is the following.
In Section \ref{secprelim} some notation and terminology is introduced. Section \ref{sec3} reviews some auxiliary results
regarding harmonic measure, most of them well known in the area. Sections \ref{sec4}-\ref{sec9} are devoted to the proof
of Theorem \ref{teo1}. The main step consists in proving the Main Lemma \ref{mainlemma}, stated in Section
\ref{sec4}. Some of the arguments to prove this (specially the ones for the Key Lemma \ref{lem:key} ) are inspired by similar techniques from \cite{AHM3TV}. 
The proof of Theorem \ref{teo1} is completed in Section \ref{sec9} by means of the Main Lemma \ref{mainlemma} and 
 a corona type decomposition valid for non-doubling measures. Some analogous corona type decompositions have already appeared in works such as \cite{Tolsa-bilip} and 
 \cite{AT}. 
 
 Section \ref{sec10} is 
 devoted to the study of harmonic measure on uniform domains and the application of the obtained results (such as Theorem 
 \ref{teounif}) to the 
 proof of Theorem \ref{teo2}. A basic ingredient for our results on harmonic measure in these domains is the 
 boundary Harnack principle of Aikawa \cite{Ai1}. Finally, Section \ref{sec11} deals with the situation when $\mu$ is
 assumed to be AD-regular.

\vvv

{\bf Acknowledgement.} We would like to thank Jonas Azzam for very helpful discussions in connection with this paper.

\vv

\section{Notation and preliminaries} \label{secprelim}

\subsection{Generalities}
We will write $a\lesssim b$ if there is $C>0$ so that $a\leq Cb$ and $a\lesssim_{t} b$ if the constant $C$ depends on the parameter $t$. We write $a\approx b$ to mean $a\lesssim b\lesssim a$ and define $a\approx_{t}b$ similarly. 

We denote the open ball of radius $r$ centered at $x$ by $B(x,r)$. For a ball $B=B(x,r)$ and $\delta>0$ we write $r(B)$ for its radius and $\delta B=B(x,\delta r)$. We let $U_\ve (A)$ to be the $\ve$-neighborhood of a set $A\subset \R^{n+1}$.

\subsection{Measures and Riesz transforms}

The Lebesgue measure of a set $A\subset \R^{n+1}$ is denoted by $m(A)$. Given $0<\delta\leq\infty$, we set
\[\HH^{n}_{\delta}(A)=\inf\left\{\textstyle{ \sum_i \diam(A_i)^n: A_i\subset\R^{n+1},\,\diam(A_i)\leq\delta,\,A\subset \bigcup_i A_i}\right\}.\]
We define the {\it $n$-dimensional Hausdorff measure} as
\[\HH^{n}(A)=\lim_{\delta\downarrow 0}\HH^{n}_{\delta}(A)\]
and the {\it $n$-dimensional Hausdorff content} as $\HH^{n}_{\infty}(A)$. 


\vv
Given a signed Radon measure $\nu$ in $\R^{n+1}$ we consider the $n$-dimensional Riesz
transform
$$\RR\nu(x) = \int \frac{x-y}{|x-y|^{n+1}}\,d\nu(y),$$
whenever the integral makes sense. For $\ve>0$, its $\ve$-truncated version is given by 
$$\RR_\ve \nu(x) = \int_{|x-y|>\ve} \frac{x-y}{|x-y|^{n+1}}\,d\nu(y).$$
For a positive Radon measure $\mu$ and a function $f\in L^1_{loc}(\mu)$, we set
$$\RR_\mu f\equiv \RR(f\,\mu),\qquad \RR_{\mu,\ve} f\equiv \RR_\ve(f\,\mu).$$
We say that the Riesz transform $\RR_\mu$ is bounded in $L^2(\mu)$ if the truncated operators 
$\RR_{\mu,\ve}:L^2(\mu)\to L^2(\mu)$ are bounded uniformly on $\ve>0$.

For $\delta\geq0$
 we set
$$\RR_{*,\delta} \nu(x)= \sup_{\ve>\delta} |\RR_\ve \nu(x)|.$$
We also consider the maximal operator
$$M^n_\delta\nu(x) = \sup_{r>\delta}\frac{|\nu|(B(x,r))}{r^n},$$
In the case $\delta=0$ we write $\RR_{*} \nu(x):= \RR_{*,0} \nu(x)$ and $M^n\nu(x):=M^n_0\nu(x)$.

\subsection{Rectifiability}

A set $E\subset \R^d$ is called $n$-rectifiable if there are Lipschitz maps
$f_i:\R^n\to\R^d$, $i=1,2,\ldots$, such that 
\begin{equation}\label{eq001}
\HH^n\biggl(E\setminus\bigcup_i f_i(\R^n)\biggr) = 0,
\end{equation}
where $\HH^n$ stands for the $n$-dimensional Hausdorff measure. 
Also, one says that 
a Radon measure $\mu$ on $\R^d$ is $n$-rectifiable if $\mu$ vanishes out of an $n$-rectifiable
set $E\subset\R^d$ and moreover $\mu$ is absolutely continuous with respect to $\HH^n|_E$.

A measure $\mu$ is called $n$-AD-regular (or just AD-regular or Ahlfors-David regular) if there exists some
constant $c>0$ such that
$$c^{-1}r^n\leq \mu(B(x,r))\leq c\,r^n\quad \mbox{ for all $x\in
\supp(\mu)$ and $0<r\leq \diam(\supp(\mu))$.}$$

A measure $\mu$ is  uniformly  $n$-rectifiable if it is 
$n$-AD-regular and
there exist $\theta, M >0$ such that for all $x \in \supp(\mu)$ and all $r>0$ 
there is a Lipschitz mapping $g$ from the ball $B_n(0,r)$ in $\R^{n}$ to $\R^d$ with $\text{Lip}(g) \leq M$ such that$$
\mu (B(x,r)\cap g(B_{n}(0,r)))\geq \theta r^{n}.$$
In the case $n=1$, $\mu$ is uniformly $1$-rectifiable if and only if $\supp(\mu)$ is contained in a rectifiable curve $\Gamma$ in $\R^d$ such that the arc length measure on $\Gamma$ is $1$-AD-regular.

A set $E\subset\R^d$ is called $n$-AD-regular if $\HH^n|_E$ is $n$-AD-regular, and it is called
uniformly $n$-rectifiable if $\HH^n|_E$ is uniformly  $n$-rectifiable.

\subsection{Uniform and NTA domains}

Following \cite{JK}, we say that an open set $\Omega\subset\R^{n+1}$ satisfies the ``corkscrew condition" if there exists some constant $c>0$
such that for all $\xi\in\partial\Omega$ and all $0<r < \diam(\partial\Omega)$ there
is a ball $B(x, cr) \subset B(\xi, r) \cap\Omega$. The point $x$ 
 is called a ``Corkscrew point"
relative to the ball $B(\xi,r)$. 
 
 Again as in \cite{JK}, we say that $\Omega$ satisfies the Harnack Chain condition if there is a constant $c$ such that
for every $\rho> 0$, $\Lambda\geq 1$, and every pair of points $x_1,x_2\in \Omega$ with $\dist(x_i,\partial\Omega)\geq\rho$
 for $i=1,2$ and $|x_1-x_2|<\Lambda\rho$, there is a chain of open balls $B_1,\ldots,B_N\subset\Omega$, with $N\leq C(\Lambda)$, with
$x_1\in B_1$, $x_2\in B_N$, $B_k\cap B_{k+1}\neq\varnothing$ and $\dist(B_k,\partial\Omega)\approx_c \diam(B_k)$
 for all $k$.
The preceding chain of balls is called a ``Harnack chain''.

A domain $\Omega\subset\R^{n+1}$ is called uniform if it satisfies the corkscrew and the Harnack chain conditions.
On the other hand, $\Omega$ is uniform and the exterior of $\Omega$ is non-empty and also satisfies the corkscrew condition, then $\Omega$ is called NTA (which stands for ``non-tangentially accessible").

\vv


\section{Some general estimates concerning harmonic measure}\label{sec3}

The following is a classical result due Bourgain. For the proof of this in the precise way it is stated below, see
\cite{AMT} or \cite{AHM3TV}.

\begin{lemma}
\label{lembourgain}
There is $\delta_{0}\in(0,1)$ depending only on $n\geq 1$ so that the following holds for $0<\delta\leq \delta_0$. Let $\Omega\subsetneq \R^{n+1}$ be a domain, $\xi \in \partial \Omega$, $r>0$, $B=B(\xi,r)$.
For all $s>n-1$ we have 
\[ \omega_{\Omega}^{x}(B)\gtrsim_{s} \frac{\mathcal H_\infty^{s}(\partial\Omega\cap \delta B)}{(\delta r)^{s}}\quad\mbox{  for all }x\in \delta B\cap\Omega.\]
\end{lemma}

\begin{rem}
If $\mu$ is some measure supported on $\d \Omega$ such that $\mu(B(x,r))\leq C\,r^n$, from the preceding lemma 
we deduce that
\begin{equation}\label{eq:BourFro}
 \omega_{\Omega}^{x}(B)\gtrsim  \frac{\mu(\partial\Omega\cap \delta B)}{(\delta r)^{n}}\quad\mbox{  for all }x\in \delta B\cap\Omega.
\end{equation}
\end{rem}

For a Greenian open set, we may write the Green function as follows (see \cite[Lemma 4.5.1]{Hel}): \begin{equation}\label{green}
G(x,y) = \mathcal{E}(x-y) - \int_{\partial\Omega} \mathcal{E}(x-z)\,d\omega^y(z), 
\quad
\mbox{for $x,y\in\Omega$, $x\neq y$,}
\end{equation}
where
 $\EE$ denotes the fundamental solution of Laplace's equation in $\R^{n+1}$, so that $\mathcal{E}(x)=c_n\,|x|^{1-n}$ for $n\geq 2$, and 
$\EE(x)=-c_1\,\log|x|$ for $n=1$, $c_1, c_n>0$.

For  $x\in\R^{n+1}\setminus \Omega$ and $y\in\Omega$, we will also set 
\begin{equation}\label{green2}
G(x,y)=0.
\end{equation}

\vv
The next result is proved in \cite{AHM3TV} too.

\begin{lemma}\label{lemgreen*}
Let $\Omega$ be a Greenian domain and let $y\in\Omega$. For $m$-almost all $x\in\Omega^c$ we have
\begin{equation}\label{eqdf12}
\mathcal{E}(x-y) - \int_{\partial\Omega} \mathcal{E}(x-z)\,d\omega^y(z)=0.
\end{equation}
\end{lemma}

\begin{rem}\label{remgreen*}
As a corollary of the preceding lemma we deduce that
$$G(x,y) = \mathcal{E}(x-y) - \int_{\partial\Omega} \mathcal{E}(x-z)\,d\omega^y(z)\quad \mbox{for 
$m$-a.e. $x\in\R^{n+1}$ and all $y\in\Omega$.}$$
\end{rem}

We will also need the following auxiliary result, which follows by standard arguments involving the maximum principle.
For the proof, see \cite{HMMTV} or  \cite{AHM3TV}.

\begin{lemma}\label{l:w>G}
Let $n\ge 2$ and $\Omega\subset\R^{n+1}$ be a bounded open connected set.
Let $B=B(x,r)$ be a closed ball with $x\in\pom$ and $0<r<\diam(\pom)$. Then, for all $a>0$,
\begin{equation}\label{eq:Green-lowerbound}
 \omega^{x}(aB)\gtrsim \inf_{z\in 2B\cap \Omega} \omega_\Omega^{z}(aB)\, r^{n-1}\, G(x,y)\quad\mbox{
 for all $x\in \Omega\backslash  2B$ and $y\in B\cap\Omega$,}
 \end{equation}
 with the implicit constant independent of $a$.
\end{lemma}

\vv

\section{The Main Lemma}\label{sec4}

Given a fixed Radon measure $\mu$, 
we say that a ball $B$ has $C_1$-thin boundary (or just thin boundary) if
\begin{equation}\label{eqthin4}
\mu\bigl(\{x\in 2B:\dist(x,\partial B)\leq t\,r(B)\}\bigr)\leq C_1\,t\,\mu(2B)\quad \mbox{ for all $t\in(0,1)$.}
\end{equation}

\begin{mlemma}\label{mainlemma}
Let $n\geq 1$, $\Omega$ be an open set in $\R^{n+1}$
  and $\mu$ be a Radon measure supported on $\d \om$ and such that $\mu(B(x,r)) \leq C_{\mu}\, r^n$, for every $x \in \d\om$ and $r>0$. For some $C_1,C_2\geq1$,
let $B\subset \R^{n+1}$ be a ball with $C_1$-thin boundary centered at $\supp\mu$ 
such that $\mu(2B)\leq C_2\,\mu(\frac{\delta_0}2B)$, where $\delta_0$ is the constant in Lemma \ref{lembourgain}. Suppose that there exist $x_B\in\frac{\delta_0}2 B\cap\Omega$ and
$\ve, \ve' \in (0,1)$ such that for any subset $E \subset B$, 
 \begin{equation}\label{eq:Amu10}
\text{if }\quad \mu(E)\leq \ve\,\mu(B),\quad\text{ then }\quad
\hm^{x_B}(E)\leq \ve'\,\hm^{x_B}(B).
\end{equation} 
Then, for every $\eta\in(0,\frac1{10})$, one of the following alternatives holds:
\begin{itemize}
\item[(i)] Either
$$\mu(B(x_B,\eta\,r(B)))\geq \tau\,\mu(B),$$
where $\tau$ is some positive constant depending on  $C_\mu$, $\ve$, $\ve'$, $C_1$ and $C_2$ (but not on $\eta$); 
or
\vv

\item[(ii)] there exists some subset $G\subset B$ with $\mu(G)\geq \theta\mu(B)$, $\theta>0$, such that the
 Riesz transform $\RR_{\mu|_G}: L^2(\mu|_G) \to L^2(\mu|_G)$ is bounded. The constant $\theta$ and the 
 $L^2(\mu|_G)$ norm depend only on $C_\mu$, $\ve$, $\ve'$, $C_1$, $C_2$, and $\eta$.
\end{itemize} 
\end{mlemma}

From now on, we assume that the constant $\kappa$ from Theorem \ref{teo1} is
$$\kappa= \frac{\delta_0}2.$$

\vv
The first step for the proof of the Main Lemma is the following.

\begin{lemma}\label{lembola1}
Let $\Omega$, $\mu$, and $B$ be as in the Main Lemma \ref{mainlemma}. 
Let $\lambda=1-\frac\ve{2C_1 C_2}$.
The ball $B_0=\lambda B$ is $\mu$-$(2,2C_2)$-doubling, $\omega^{x_B}$-$(\lambda^{-1},(1-\ve)^{-1})$-doubling, and
satisfies the following:
for any subset $E \subset B_0$,  
 \begin{equation}\label{eq:Amu100}
\text{if }\quad \mu(E)\leq \frac{\ve}2\,\mu(B_0),\quad \text{ then }\quad \hm^{x_B}(E)\leq \ve'\,\hm^{x_B}(B_0).
\end{equation} 
\end{lemma}

Note that in the preceding lemma, the pole for harmonic measure is $x_B$, the same as for the ball $B$.
Observe also that $\lambda\in (1/2,1)$ and thus
$$
\frac12 B\subset B_0\subset B.
$$
Since $\mu(B)\leq\mu(2B)\leq C_2\,\mu(\frac{\delta_0}2B)$ and $\delta_0\leq1$, we have
 \begin{equation}\label{eq:Amu2001}
 \mu(B)\leq C_2\,\mu(B_0).
\end{equation}
Note also that, by taking complements, the assertion \rf{eq:Amu100} implies that
 \begin{equation}\label{eq:Amu200}
\text{if }\quad \hm^{x_B}(E)< (1-\ve')\,\hm^{x_B}(B_0),\quad \text{ then }\quad \mu(E)< (1-\ve/2)\mu(B_0).
\end{equation}
\vv

\begin{proof}[Proof of Lemma \ref{lembola1}]
From the thin boundary property and the doubling condition, we deduce that
\begin{equation}\label{eqme456}
\mu(B\setminus \lambda B)\leq C_1(1-\lambda)\mu(2B)\leq  C_1 \,C_2\,(1-\lambda)\mu(B)= \frac{\ve}2\,\mu(B).
\end{equation}
This implies that 
$$\mu(\lambda B) =\mu(B) - \mu(B\setminus \lambda B)\geq \left(1-\frac\ve2\right)\,\mu(B)\geq \frac{1-\frac\ve2}{C_2}\,\mu(2B)
\geq\frac{1-\frac\ve2}{C_2}\,\mu(2\lambda B)
,$$
and since $\frac{1-\frac\ve2}{C_2} \geq \frac1{2C_2}$, $B_0=\lambda B$ is $(2,2C_2)$-doubling.

From \rf{eqme456} and \rf{eq:Amu10} we deduce that
$$\omega^{x_B}(B\setminus \lambda B)\leq \ve'\,\omega^{x_B}(B) = \ve'\,\omega^{x_B}(B\setminus \lambda B) + \ve'\,\omega^{x_B}(\lambda B).$$
Thus,
$$\omega^{x_B}(B\setminus \lambda B)\leq \frac{\ve'}{1-\ve'}\,\omega^{x_B}(\lambda B),$$
and so
$$\omega^{x_B}(B)\leq \omega^{x_B}( \lambda B)+ 
\frac{\ve'}{1-\ve'}\,\omega^{x_B}(\lambda B) = \frac{1}{1-\ve'}\,\omega^{x_B}(\lambda B).$$
In other words, $B_0=\lambda\,B$ is $\omega^{x_B}$-$(\lambda^{-1},(1-\ve)^{-1})$-doubling.

To prove that for $E\subset B_0$  the condition \rf{eq:Amu100} holds, consider the auxiliary set
$$\wt E= E \cup (B\setminus \lambda B).$$
Using  \rf{eqme456}, we deduce that
$$\mu(\wt E)= \mu(E) + \mu(B\setminus \lambda B)\leq \frac\ve2\,\mu(B) + \frac\ve2\,\mu(B) = \ve\,\mu(B).$$
So from the condition \rf{eq:Amu10} we infer that
$$\omega^{x_B}(\wt E)\leq \ve'\,\omega^{x_B}(B),$$
which is equivalent to saying that
$$\omega^{x_B}(E)+ \omega^{x_B}(B\setminus \lambda B)\leq \ve'\,\omega^{x_B}(\lambda B) +  \ve'\,\omega^{x_B}(B\setminus \lambda B)
.$$
This implies that
$$\omega^{x_B}(E)\leq \ve'\,\omega^{x_B}(\lambda B),$$
as wished.
\end{proof}
\vv

\begin{lemma}\label{lemaux1}
We have
\begin{equation}\label{eq:893'}
\hm^{x_B}(B_0) \gtrsim  \frac{\mu(B_0)}{r(B_0)^{n}}.
\end{equation}
\end{lemma}

\begin{proof}
By \rf{eq:BourFro} we have
$$ \omega^{x}(B_0)\gtrsim  \frac{\mu({\delta_0} B_0)}{({\delta_0} r)^{n}}\quad\mbox{  for all }x\in {\delta_0} B_0\cap \Omega.
$$
So \rf{eq:893'} holds because $x_B\in\frac{\delta_0}2B\subset {\delta_0} B_0$ (since $B\subset 2B_0$) and
$$\mu(B_0)\leq \mu(B) \leq C_2\,\mu(\tfrac{\delta_0}2B)\leq C_2\,\mu({\delta_0} B_0).$$
\end{proof}

\vv


\section{The dyadic lattice of David and Mattila}\label{sec:DMlatt}\label{sec5}

Now we will consider the dyadic lattice of cubes
with small boundaries of David-Mattila associated with a Radon measure $\sigma$. This lattice has been constructed in \cite[Theorem 3.2]{David-Mattila}. 
Its properties are summarized in the next lemma.

\begin{lemma}[David, Mattila]
\label{lemcubs}
Let $\sigma$ be a compactly supported Radon measure in $\R^{n+1}$.
Consider two constants $C_0>1$ and $A_0>5000\,C_0$ and denote $W=\supp\sigma$. 
Then there exists a sequence of partitions of $W$ into
Borel subsets $Q$, $Q\in \DD_{\sigma,k}$, with the following properties:
\begin{itemize}
\item For each integer $k\geq0$, $W$ is the disjoint union of the ``cubes'' $Q$, $Q\in\DD_{\sigma,k}$, and
if $k<l$, $Q\in\DD_{\sigma,l}$, and $R\in\DD_{\sigma,k}$, then either $Q\cap R=\varnothing$ or else $Q\subset R$.
\vv

\item The general position of the cubes $Q$ can be described as follows. For each $k\geq0$ and each cube $Q\in\DD_{\sigma,k}$, there is a ball $B(Q)=B(z_Q,r(Q))$ such that
$$z_Q\in W, \qquad A_0^{-k}\leq r(Q)\leq C_0\,A_0^{-k},$$
$$W\cap B(Q)\subset Q\subset W\cap 28\,B(Q)=W \cap B(z_Q,28r(Q)),$$
and
$$\mbox{the balls\, $5B(Q)$, $Q\in\DD_{\sigma,k}$, are disjoint.}$$

\vv
\item The cubes $Q\in\DD_{\sigma,k}$ have small boundaries. That is, for each $Q\in\DD_{\sigma,k}$ and each
integer $l\geq0$, set
$$N_l^{ext}(Q)= \{x\in W\setminus Q:\,\dist(x,Q)< A_0^{-k-l}\},$$
$$N_l^{int}(Q)= \{x\in Q:\,\dist(x,W\setminus Q)< A_0^{-k-l}\},$$
and
$$N_l(Q)= N_l^{ext}(Q) \cup N_l^{int}(Q).$$
Then
\begin{equation}\label{eqsmb2}
\sigma(N_l(Q))\leq (C^{-1}C_0^{-3(n+1)-1}A_0)^{-l}\,\sigma(90B(Q)).
\end{equation}
\vv

\item Denote by $\DD_{\sigma,k}^{db}$ the family of cubes $Q\in\DD_{\sigma,k}$ for which
\begin{equation}\label{eqdob22}
\sigma(100B(Q))\leq C_0\,\sigma(B(Q)).
\end{equation}
We have that $r(Q)=A_0^{-k}$ when $Q\in\DD_{\sigma,k}\setminus \DD_{\sigma,k}^{db}$
and
\begin{equation}\label{eqdob23}
\sigma(100B(Q))\leq C_0^{-l}\,\sigma(100^{l+1}B(Q))\quad
\mbox{for all $l\geq1$ with $100^l\leq C_0$ and $Q\in\DD_{\sigma,k}\setminus \DD_{\sigma,k}^{db}$.}
\end{equation}
\end{itemize}
\end{lemma}

\vv

We use the notation $\DD_\sigma=\bigcup_{k\geq0}\DD_{\sigma,k}$. Observe that the families $\DD_{\sigma,k}$ are only defined for $k\geq0$. So the diameter of the cubes from $\DD$ are uniformly
bounded from above.
We set
$\ell(Q)= 56\,C_0\,A_0^{-k}$ and we call it the side length of $Q$. Notice that 
$$\frac1{28}\,C_0^{-1}\ell(Q)\leq \diam(28B(Q))\leq\ell(Q).$$
Observe that $r(Q)\approx\diam(Q)\approx\ell(Q)$.
Also we call $z_Q$ the center of $Q$, and the cube $Q'\in \DD_{\sigma,k-1}$ such that $Q'\supset Q$ the parent of $Q$.
 We set
$B_Q=28 B(Q)=B(z_Q,28\,r(Q))$, so that 
$$W\cap \tfrac1{28}B_Q\subset Q\subset B_Q.$$

We assume $A_0$ big enough so that the constant $C^{-1}C_0^{-3(n+1)-1}A_0$ in 
\rf{eqsmb2} satisfies 
$$C^{-1}C_0^{-3(n+1)-1}A_0>A_0^{1/2}>10.$$
Then we deduce that, for all $0<\lambda\leq1$,
\begin{align}\label{eqfk490}\nonumber
\sigma\bigl(\{x\in Q:\dist(x,W\setminus Q)\leq \lambda\,\ell(Q)\}\bigr) + 
\sigma\bigl(\bigl\{x\in 3.5B_Q:\dist&(x,Q)\leq \lambda\,\ell(Q)\}\bigr)\\
&\leq
c\,\lambda^{1/2}\,\sigma(3.5B_Q).
\end{align}

We denote
$\DD_\sigma^{db}=\bigcup_{k\geq0}\DD_{\sigma,k}^{db}$.
Note that, in particular, from \rf{eqdob22} it follows that
\begin{equation}\label{eqdob*}
\sigma(3B_{Q})\leq \sigma(100B(Q))\leq C_0\,\sigma(Q)\qquad\mbox{if $Q\in\DD_\sigma^{db}.$}
\end{equation}
For this reason we will call the cubes from $\DD_\sigma^{db}$ doubling. 
Given $Q\in\DD_\sigma$, we denote by $\DD_\sigma(Q)$
the family of cubes from $\DD_\sigma$ which are contained in $Q$. Analogously,
we write $\DD_\sigma^{db}(Q) = \DD^{db}_\sigma\cap\DD(Q)$.

As shown in \cite[Lemma 5.28]{David-Mattila}, every cube $R\in\DD_\sigma$ can be covered $\sigma$-a.e.\
by a family of doubling cubes:
\vv

\begin{lemma}\label{lemcobdob}
Let $R\in\DD_\sigma$. Suppose that the constants $A_0$ and $C_0$ in Lemma \ref{lemcubs} are
chosen suitably. Then there exists a family of
doubling cubes $\{Q_i\}_{i\in I}\subset \DD_\sigma^{db}$, with
$Q_i\subset R$ for all $i$, such that their union covers $\sigma$-almost all $R$.
\end{lemma}

The following result is proved in \cite[Lemma 5.31]{David-Mattila}.
\vv

\begin{lemma}\label{lemcad22}
Let $R\in\DD_\sigma$ and let $Q\subset R$ be a cube such that all the intermediate cubes $S$,
$Q\subsetneq S\subsetneq R$ are non-doubling (i.e.\ belong to $\DD_\sigma\setminus \DD_\sigma^{db}$).
Then
\begin{equation}\label{eqdk88}
\sigma(100B(Q))\leq A_0^{-10n(J(Q)-J(R)-1)}\sigma(100B(R)).
\end{equation}
\end{lemma}


Given a ball (or an arbitrary set) $B\subset \R^{n+1}$, we consider its $n$-dimensional density:
$$\Theta_\sigma(B)= \frac{\sigma(B)}{\diam(B)^n}.$$

From the preceding lemma we deduce:

\vv
\begin{lemma}\label{lemcad23}
Let $Q,R\in\DD_\sigma$ be as in Lemma \ref{lemcad22}.
Then
$$\Theta_\sigma(100B(Q))\leq C_0\,A_0^{-9n(J(Q)-J(R)-1)}\,\Theta_\sigma(100B(R))$$
and
$$\sum_{S\in\DD_\sigma:Q\subset S\subset R}\Theta_\sigma(100B(S))\leq c\,\Theta_\sigma(100B(R)),$$
with $c$ depending on $C_0$ and $A_0$.
\end{lemma}

For the easy proof, see
 \cite[Lemma 4.4]{Tolsa-memo}, for example.

\vv


\section{Good and bad collections of cubes from $\DD_\omega$}\label{secbad}\label{sec6}

\subsection{Definition of good and bad cubes}

From now on, $B$ and $B_0$ are the balls in Main Lemma \ref{mainlemma} and Lemma \ref{lembola1}.
To simplify notation, we denote $\alpha=\lambda^{-1}$, so that $B_0$ is $\omega^{x_B}$-$(\alpha,(1-\ve')^{-1})$-doubling.
We consider the dyadic lattice of Lemma \ref{lemcubs} associated with the measure
$\sigma = \omega^{x_B}|_{10B_0}$, and we denote this by $\DD_\omega$, to shorten notation.

We now need to define a family of bad cubes.
We say that $Q\in\DD_\omega$ is {\bf bad} and we write $Q\in\bad$, if $Q\in\DD_\omega$ is a maximal cube which is contained in 
$B\equiv \alpha B_0$ satisfying one of the conditions below: 
\begin{align}
\frac{\hm^{x_B}(Q)}{\hm^{x_B}(B_0)} &\leq A^{-1} \frac{\mu(Q)}{\mu(B_0)},\label{eq:hm-stop}\\ 
\frac{\mu(Q)}{\mu(B_0)} &\leq A^{-1} \frac{\hm^{x_B}(Q)}{\hm^{x_B}(B_0)},\label{eq:mu-stop}
\end{align}
where $A$ is some big constant to be chosen below. 
If the condition \rf{eq:hm-stop} holds, we write $Q\in\bad_1$ and in the case \rf{eq:mu-stop}, $Q\in\bad_2$. Thefore, $\bad=\bad_1 \cup \bad_2$.  

We say that $Q\in\DD_\omega$ is {\bf good}, and we write $Q\in\good$ if $Q$ is contained in $\alpha B_0$ and $Q$ is not contained in any
cube from the family $\bad$.

\subsection{Packing conditions}

Abusing notation, below we
write $\bad_i$ instead of $\bigcup_{Q\in\bad_i}Q$.
Notice that, using the definition of $\bad_1$, $\bad_2$, and the doubling properties of $\mu$ and $\omega^{x_B}$, 
\begin{align}
\hm^{x_B}(\bad_1)  &\leq A^{-1} \frac{\mu(\bad_1)}{\mu(B_0)}\, \hm^{x_B}(B_0) \leq  A^{-1} \frac{\mu(\alpha B_0)}{\mu(B_0)}\,\hm^{x_B}(B_0)
\leq  C\,A^{-1}\hm^{x_B}(B_0)
, \label{eq:bad1-hm} \\
\mu(\bad_2) &\leq A^{-1} \frac{\hm^{x_B}(\bad_2)}{\hm^{x_B}(B_0)}\,\mu(B_0) \leq  A^{-1} \frac{\hm^{x_B}(\alpha B_0)}{\hm^{x_B}(B_0)}\,\mu(B_0)\leq C(\ve')\,A^{-1}\mu(B_0).\label{eq:bad2-mu}
\end{align}

In view of \rf{eq:Amu100} and \rf{eq:Amu200}, if $A$ is large enough, there exist $\ve_1, \ve_2 \in (0,1)$ such that 
\begin{align}
\mu(\bad_1\cap B_0) &<\ve_1\,\mu(B_0)\label{eq:bad1-mu},\\
\hm^{x_B}(\bad_2 \cap B_0)  &< \ve_2\,  \hm^{x_B}(B_0).\label{eq:bad2-hm}
\end{align}
Combining \eqref{eq:bad1-hm}, \eqref{eq:bad2-mu}, \eqref{eq:bad1-mu} and \eqref{eq:bad2-hm} we obtain that
\begin{align*}
\hm^{x_B}(\bad \cap B_0)  &< (c_{\hm} A^{-1}+\ve_2)\,  \hm^{x_B}(B_0),\\
\mu(\bad\cap B_0) &<(c_{\mu} A^{-1} +\ve_1)\,\mu(B_0).
\end{align*}
Choose now $A$ so large that $c_{\mu} A^{-1} +\ve_1=1-\ve_1'$ and $c_{\hm} A^{-1}+\ve_2=1-\ve_2'$, for some $\ve'_1, \ve'_2 \in (0,1)$. If we set $G_0:= B_0  \setminus \bigcup_{Q\in\bad} Q$, we deduce that
\begin{align}
\hm^{x_B}(G_0)& =\hm^{x_B}(B_0 \setminus \bad)\geq \ve'_2\, \hm^{x_B}(B_0 )\label{eq:bad-hm}
\end{align}
and also that 
\begin{align}
\mu(G_0) & =\mu(B_0 \setminus \bad)\geq \ve'_1\,\mu(B_0 ).\label{eq:bad-mu}
\end{align}

Notice that by Lebesgue's differentiation theorem, \eqref{eq:hm-stop}, and \eqref{eq:mu-stop} we have that
\begin{equation}\label{eqpoisson}
A^{-1} \frac{\hm^{x_B}(B_0 )}{\mu(B_0 )}\leq \frac{d\hm^{x_B}}{d\mu}(x) \leq A \frac{\hm^{x_B}(B_0)}{\mu(B_0 )}
\quad\mbox{ for $\mu$-a.e. $x \in G_0$,}
\end{equation}
and also
\begin{equation}\label{eqpoisson'}
A^{-1} \frac{\mu(B_0 )}{\hm^{x_B}(B_0 )}\leq \frac{d\mu}{d\hm^{x_B}}(x) \leq A \frac{\mu(B_0 )}{\hm^{x_B}(B_0 )}
\quad\mbox{ for $\hm^{x_B}$-a.e. $x\in G_0$.}
\end{equation}
We can think of $\frac{d\hm^{x_B}}{d\mu}=:k^{x_B}$ as the Poisson kernel with respect to $\mu$ with pole at $x_B$. What we just proved is that $k^{x_B}$ is bounded from above and away from zero in $G_0$ apart from a set of $\mu$-measure zero. 
\vv

\subsection{The growth of $\omega^{x_B}$ on the good cubes}

\begin{lemma}\label{lemgrowth1}
If $Q\in\DD_\omega\cap\good$, $100B(Q)\subset\alpha B_0$, and $Q\cap B_0\neq\varnothing$, then 
\begin{equation}\label{eqgrowth}
\omega^{x_B}(100B(Q))\leq C\, \dfrac{\hm^{x_B}(B_0)}{\mu(B_0)}\,\ell(Q)^n.
\end{equation}
\end{lemma}

\begin{proof}
Suppose first that $Q\in\DD_\omega^{db}$. Then, using also that $Q$ is good, 
$$\omega^{x_B}(100B(Q))\leq C\,\omega^{x_B}(Q)\leq C\,A\,\mu(Q)\,\frac{\omega^{x_B}(B_0)}{\mu(B_0)},$$
and by the polynomial growth of $\mu$ \rf{eqgrowth} follows.

Suppose now that $Q\not\in\DD_\omega^{db}$. Let $Q'$ be the cube from $\DD_\omega^{db}$ with minimal side length that contains $Q$.
If $Q'\subset \alpha \,B_0$, then $Q'\in\good$ and we have already shown that \rf{eqgrowth} holds for $Q'$. Thus, by Lemma \ref{lemcad23} and \rf{eq:hm-stop}, we get 
$$\Theta_{\omega^{x_B}}(100B(Q))\leq C\,\Theta_{\omega^{x_B}}(100B(Q'))\leq C\,
\frac{\omega^{x_B}(Q')}{\ell(Q')^n}\lesssim A\,\frac{\mu(Q')}{\ell(Q')^n}\,\frac{\omega^{x_B}(B_0)}{\mu(B_0)}
\lesssim A\,\frac{\omega^{x_B}(B_0)}{\mu(B_0)},$$
and so \rf{eqgrowth} also holds.

Suppose now that there is not any cube $Q'\in\DD_\omega^{db}$ such that $Q\subset Q'\subset \alpha\,B_0$. 
Then denote by $Q''$ 
the cube containing $Q$ which has maximal side length such that $100B(Q'')$ is contained in $\alpha B_0$. It turns out that
$\ell(Q'')\approx_\alpha r(B_0)$ (for this we use the fact that $\alpha>1$ and that $Q\cap B_0\neq\varnothing$).
Then we deduce that 
$$\Theta_{\omega^{x_B}}(100B(Q''))\leq C\,\Theta_{\omega^{x_B}}(B_0).$$
Then applying Lemma \ref{lemcad23} again,
$$\Theta_{\omega^{x_B}}(100B(Q))\leq C\,\Theta_{\omega^{x_B}}(100B(Q''))\leq 
C\,\Theta_{\omega^{x_B}}(B_0),$$
and hence \rf{eqgrowth} also holds in this case.
\end{proof}

From Lemma \ref{lemgrowth1} we easily get the following.

\begin{lemma}\label{lemgrowth2}
If $Q\in\DD_\omega\cap\good$, $Q\subset\alpha B_0$, and $Q\cap B_0\neq\varnothing$, then
$$\omega^{x_B}(B(x,r))\leq C\, \dfrac{\hm^{x_B}(B_0)}{\mu(B_0)}\,r^n\quad\mbox{ for all $x\in Q$ and $r\geq \ell(Q)$.}$$
\end{lemma}

\begin{proof}
Notice first that, by Lemma \ref{lemaux1}, any ball $B(x,r)$ with $r\gtrsim r(B_0)$ satisfies
\begin{equation}\label{eqdfa1}
\omega^{x_B}(B(x,r)) \leq 1\lesssim \frac{\hm^{x_B}(B_0)}{\mu(B_0)}\,{r(B_0)^{n}}\lesssim \frac{\hm^{x_B}(B_0)}{\mu(B_0)}\,r^n.
\end{equation}

Suppose now that $r\leq c\, r(B_0)$ for small $c>0$.
Let
 $R\in\DD_\omega$ be the smallest  cube containing $Q$ such that $B(x,r)\subset 100B(R)$, so that moreover 
 $r\approx \ell(R)$ and $R\cap B_0\neq\varnothing$ (because $Q\cap B_0\neq\varnothing$). 
 If $100B(R)\subset\alpha B_0$ (in particular this implies that $R\in\good$), by Lemma \ref{lemgrowth1}, 
 \begin{equation}\label{eqdfa2}
\omega^{x_B}(B(x,r))\leq\omega^{x_B}(100B(R))\lesssim \dfrac{\hm^{x_B}(B_0)}{\mu(B_0)}\,\ell(R)^n\approx  \dfrac{\hm^{x_B}(B_0)}{\mu(B_0)}\,r^n
\end{equation}
If $100B(R)\not\subset\alpha B_0$, from the fact $R\cap B_0\neq\varnothing$ we deduce that $r\approx \ell(R)\gtrsim_\alpha r(B_0)$
and so \rf{eqdfa2} also holds, because of \rf{eqdfa1}.

The lemma follows easily from the previous discussion.
\end{proof}

\vv

\section{The key lemma about the Riesz transform on good cubes}\label{sec7}

\begin{keylemma}\label{lem:key}
Let $\Omega$, $\mu$, $\eta$, $B$ and $B_0$ be as in the Main Lemma \ref{mainlemma} and Lemma \ref{lembola1}. Let also $Q \in \good$ be such that $Q\cap \bigl(B_0 \setminus B(x_B,  \eta \, r(B))\bigr)\neq\varnothing$, $100B(Q)\subset B$, ${\delta_0} r(B_Q) \leq \eta\, r(B) $ and $Q \subset \partial \Omega \setminus B(x_B,  \frac{\eta}{2} \, r(B))$. For all $z\in Q$  we have
\begin{equation}\label{eqdk0}
 \bigl|\RR_{\ell(Q)}\omega^{x_B}(z)\bigr| \lesssim \frac{\hm^{x_B}(B_0)}{\mu(B_0)},
\end{equation}
where the implicit constant depends on $c_\omega$, $\ve$, $\ve'$, $C_1$, $C_2$, $A$ and $\eta$.
\end{keylemma}

\vspace{1mm}
\begin{proof}[\bf Proof in the case \boldmath{$n\geq2$}] 

Let $\vphi:\R^d\to[0,1]$ be a radial $\CC^\infty$ function  which vanishes on $B(0,1)$ and equals $1$ on $\R^d\setminus B(0,2)$,
and for $\ve>0$ and $z\in \R^{n+1}$ denote
$\vphi_\ve(z) = \vphi\left(\frac{z}\ve\right) $ and $\psi_\ve = 1-\vphi_\ve$.
We set
$$\wt\RR_\ve\omega^{x_B}(z) =\int K(z-y)\,\vphi_\ve(z-y)\,d\omega^{x_B}(y),$$
where $K(\cdot)$ is the kernel of the $n$-dimensional Riesz transform. 

We consider first the case when $Q\in \DD_\omega^{db}$. Take a ball $\wt B_Q$ centered at some point of $Q$ such that $r(\wt B_Q)= \frac{{\delta_0}}{10} \,r(B_Q)$ and $\mu(\wt B_Q)\gtrsim \mu(B_Q)$, with the implicit constant depending on ${\delta_0}$. Notice that for any $x \in \wt B_Q$ we have that $|x-x_B| \geq c(\eta) \, r(B)> 2\,r(\wt B_Q)$. To shorten notation, in the rest of the proof we will write $r=r(\wt B_Q)$. 

Note that, for every $z\in Q \subset \pom$, by standard Calder\'on-Zygmund estimates 
\begin{align*}
\bigl|\wt\RR_{r}\omega^{x_B}(x) - \RR_{r(B_Q)} \omega^{x_B}(z)\bigr|&\lesssim \,\frac{\omega^{x_B}(B(x, 3\, r(B_Q))}{r^n}\\ 
&\lesssim_{\delta_0}\, \frac{\omega^{x_B}(100 B(Q))}{\mu(Q)}
\lesssim \frac{\omega^{x_B}(Q)}{\mu(Q)}
 \lesssim_A \frac{\omega^{x_B}(B_0)}{\mu(B_0)},
\end{align*}
where in the penultimate inequality we used that $Q \in \DD^{db}_\omega$ and in the last one that $Q\in \good$.

For a fixed $x\in Q\subset\pom$ and $z\in \R^{n+1}\setminus \bigl[\supp(\vphi_r(x-\cdot)\,\omega^{x_B})\cup \{x_B\}\bigr]$, consider the function
$$u_r(z) = \EE(z-x_B) - \int \EE(z-y)\,\vphi_r(x-y)\,d\omega^{x_B}(y),$$
so that, by Remark \ref{remgreen*},
\begin{equation}\label{eqfj33}
G(z,x_B) = u_r(z) - \int \EE(z-y)\,\psi_r(x-y)\,d\omega^{x_B}(y)\quad \mbox{ for $m$-a.e.  $z\in\R^{n+1}$.}
\end{equation}

Since the kernel of the Riesz transform is
\begin{equation}\label{eqker}
K(x) = c_n\,\nabla \EE(x),
\end{equation}
for a suitable absolute constant $c_n$, we have
$$\nabla u_r(z) = c_n\,K(z-x_B) - c_n\,\RR(\vphi_{ r}(\cdot-x)\,\omega^{x_B})(z).$$

In the particular case $z=x$ we get
$$\nabla u_r(x) = c_n\,K(x-x_B) - c_n\,\wt\RR_r\omega^{x_B}(x),$$
and thus
\begin{equation}\label{eqcv1}
|\wt\RR_r\omega^{x_B}(x)|\lesssim \frac1{|x-x_B|^n} + |\nabla u_r(x)|.
\end{equation}
Observe that, by Lemma \ref{lemaux1},
$$\frac1{|x-x_B|^n}\lesssim \frac{C(\eta)}{r(B_0)^n} \lesssim_\eta \frac{\hm^{x_B}(B_0)}{\mu(B_0)}.$$

Now we deal with the last summand in \rf{eqcv1}.
Since $u_r$ is harmonic in $\R^{n+1}\setminus \bigl[\supp(\vphi_r(x-\cdot)\,\omega^{x_B})\cup \{{x_B}\}\bigr]$ (and so in $B(x,r)$),  
we have
\begin{equation}\label{eqcv2}
|\nabla u_r(x)| \lesssim \frac1r\,\avint_{B(x,r)}|u_r(z)|\,dm(z).
\end{equation}
From the identity \rf{eqfj33} we deduce that
\begin{align}\label{eqcv3}
|\nabla u_r(x)| &\lesssim \frac1r\,\avint_{B(x,r)}G(z,{x_B})\,dm(z) + 
\frac1r\,\avint_{B(x,r)}
\int \EE(z-y)\,\psi_r(x-y)\,d\omega^{x_B}(y)\,dm(z) \nonumber\\
& =:I + II.
\end{align}
To estimate the term $II$ we use Fubini and the fact that $\supp\psi_r\subset B(x,2r)$:
\begin{align*}
II & \lesssim \frac1{r^{n+2}}\, \int_{y\in B(x,2r)}\int_{z\in B(x,r)} \frac 1{|z-y|^{n-1}} \,dm(z)\,d\omega^{x_B}(y)\\
& \lesssim \frac{\omega^{x_B}(B(x,2r))}{r^{n}} \lesssim \frac{\omega^{{x_B}}(3B_Q)}{\mu(Q)} \lesssim_A  \dfrac{\hm^{x_B}(B_0)}{\mu(B_0)},
\end{align*}
where the last inequality follows from the fact that $Q \in \DD^{db}_\omega \cap \good$. We intend to show now that $I\lesssim  \dfrac{\hm^{x_B}(B_0)}{\mu(B_0)}.$
Clearly it is enough to show that
\begin{equation}\label{eqsuf1}
\frac1r\,| G(y,x_B)|\lesssim   \dfrac{\hm^{x_B}(B_0)}{\mu(B_0)}\qquad\mbox{for all $y\in  B(x,r)\cap\Omega$.}
\end{equation}
To prove this, observe that by Lemma \ref{l:w>G} (with $B= B(x,r)$, $a=2{\delta_0}^{-1}$), for all $y\in B(x,r)\cap\Omega$,
we have
$$\omega^{x_B}(B(x,2{\delta_0}^{-1}r))\gtrsim \inf_{z\in B(x,2r)\cap \Omega} \omega^{z}(B(x,2{\delta_0}^{-1}r))\, r^{n-1}\,|G(y,x_B)|.$$
On the other hand, by Lemma \ref{lembourgain}, for any $z\in B(x,2r)\cap\Omega$,
$$\omega^{z}(B(x,2{\delta_0}^{-1}r))\gtrsim \frac{\mu(B(x,2r))}{r^n}\geq \frac{\mu(\wt B_Q)}{r^n}.$$
Therefore we have
$$
\omega^{{x_B}}(B(x,2{\delta_0}^{-1}r))
\gtrsim 
\frac{\mu(\wt B_Q)}{r^n}\, r^{n-1}\,|G(y,{x_B})|,
$$
and thus
$$
\frac1r\,| G(y,{x_B})|\lesssim \frac{\omega^{{x_B}}(B(x,2{\delta_0}^{-1}r))}{\mu(\wt B_Q)} .
$$
Now, recall that by construction $\mu(\wt B_Q)\gtrsim \mu(B_Q)\geq \mu(Q)$ and
$B(x,2{\delta_0}^{-1}r)=2{\delta_0}^{-1} \wt B_Q\subset 3B_Q$, since $r(\wt B_Q)=\frac{\delta_0}{10}r(B_Q)$ and since $Q \in \DD^{db}_\omega \cap \good$, we have
$$
\frac1r\,| G(y,{x_B})|\lesssim 
\frac{\omega^{{x_B}}(B(x,2{\delta_0}^{-1}r))}{\mu(\wt B_Q)}\lesssim \frac{\omega^{{x_B}}(3B_Q)}{\mu(Q)} \lesssim_A  \dfrac{\hm^{x_B}(B_0)}{\mu(B_0)}.$$
So \rf{eqsuf1} is proved and the proof of the Key lemma is complete in the case $n\geq2$, $Q\in\DD^{db}_\omega$.

\vv

Consider now the case $Q\in\good\setminus\DD_\omega^{db}$. Let $Q'\in\DD_\omega^{db}$ be the cube with minimal side length such that
$Q\subset Q'\subset\alpha B_0\setminus B(x_B,\frac\eta2 r(B))$.
If such cube does not exist, we let $Q'\in\DD_\omega$ be the largest cube such that $Q\subset Q'\subset \alpha B_0\setminus B(x_B,\frac\eta2 r(B))$, so that $\ell(Q') \approx r(B_0)$ (because $Q'\cap \bigl(B_0\setminus B(x_B,\eta r(B))\bigr)\neq \varnothing$).
For all $z\in Q$ then we have
\begin{equation}\label{eqdk3902}
|\RR_{\ell(Q)}\omega^{x_B} (z)| \leq |\RR_{\ell(Q')}\omega^{x_B} (z)| + C\,\sum_{P\in\DD_\omega: Q\subset P\subset Q'} \frac{\omega^{x_B}(100B(P))}{\ell(P)^n}. 
\end{equation}
In any case, the first term on the right hand side is bounded by some constant multiple of $\frac{\omega^{x_B}(B_0)}{\mu(B_0)}$. 
This has already been shown if $Q'\in\DD_\omega^{db}$, while in the case $Q'\notin\DD_\omega^{db}$, since $\ell(Q') \approx r(B_0)$ we have
$$|\RR_{\ell(Q')}\omega^{x_B} (x)|\lesssim \frac{\|\omega^{x_B}\|}{\ell(Q')^n}\lesssim \frac1{r(B_0)^n}\lesssim 
\frac{\omega^{x_B}(B_0)}{\mu(B_0)},$$
by Lemma \ref{lemaux1}.

To bound the last sum in \rf{eqdk3902}, we first notice that every $P \in \DD_\omega$ such that $Q \subset P \subset Q'$ is in $\DD_\omega\setminus\DD_\omega^{db}$ and thus, by  Lemma \ref{lemcad23}, we obtain
$$\sum_{P\in\DD_\omega: Q\subset P\subset Q'} \frac{\omega^{x_B}(100B(P))}{\ell(P)^n} \lesssim\frac{\omega^{x_B}(100B(Q'))}{\ell(Q')^n}
.$$ 
Since $Q'$ satisfies the assumptions of Lemma \ref{lemgrowth1}, by \rf{eqgrowth} we have
$$\frac{\omega^{x_B}(100B(Q'))}{\ell(Q')^n}\lesssim
 \dfrac{\hm^{x_B}(B_0)}{\mu(B_0)}.$$
So \rf{eqdk0} also holds for $Q\in\DD_\omega\setminus\DD_\omega^{db}$.
\end{proof}
\vv

\vvv

\begin{proof}[\bf Proof of the Key Lemma in the planar case \boldmath{$n=1$}]

We note  that the arguments to prove Lemma \ref{l:w>G} fail in the planar case. Therefore this cannot be
applied to prove the Key Lemma and some changes are required. 

We follow the same scheme and notation as in the case $n\geq2$ and highlight the important modifications. 
We start by assuming that $Q \in \DD^{db}_\omega$ and claim that for any constant $\alpha\in \R$,
\begin{equation}\label{eq1**}
\bigl|\wt\RR_r\omega^{x_B}(x)\bigr|\lesssim \frac{1}{r}
\avint_{B(x,r)} |G(y,{x_B}) - \alpha|\, 
\,dm(y) + \frac1{|x-{x_B}|} + \frac{\omega^{x_B}(Q)}{\mu(Q)}.
\end{equation}
To check this, we can argue as in the proof of the Key Lemma for $n\geq2$ to get\begin{equation}\label{eq2**}
|\wt\RR_r\omega^p(x)| \lesssim \frac1{|x-{x_B}|}  + |\nabla u_r(x)|\lesssim _\eta \frac{\hm^{x_B}(B_0)}{\mu(B_0)}.
\end{equation}

Since $u_r$ is harmonic in $\R^{2}\setminus \bigl[\supp(\vphi_r(x-\cdot)\,\omega^{x_B})\cup \{{x_B}\}\bigr]$ (and so in $B(x,r)$), 
for any constant $\alpha'\in\R$, we have
$$
|\nabla u_r(x)| \lesssim \frac1r\,\avint_{B(x,r)}|u_r(z)-\alpha'|\,dm(z).
$$
Note that this estimate is the same as the one in in \rf{eqcv2} in the case $n\geq2$ with  $\alpha'=0$. Let $\alpha'=\alpha +\beta\int \psi_{r}(x-y)d\omega^{x_B}(y)$ where $\beta=\avint_{B(x,r)}\EE(x-z)dm(z)$.
From the identity \rf{eqfj33}, we deduce that
\begin{align}\label{eqcv3'}
|\nabla u_r(x)| &\lesssim \frac1r\,\avint_{B(x,r)}|G(z,{x_B})-\alpha|\,dm(z)  \nonumber\\
&\quad + 
\frac1r\,\avint_{B(x,r)}
\int |\EE(z-y)-\beta|\,\psi_r(x-y)\,d\omega^{x_B}(y) \,dm(z) \nonumber\\
& =:I + II,
\end{align}
for any $\alpha\in\R$.

To estimate the term $II$ we apply Fubini:
$$II\leq  \frac cr\,
\int_{y\in B(x,2r)} \avint_{z\in B(x,r)}\left|\EE(z-y)-\beta\right|\,dm(z)\,d\omega^p(y).$$
Observe that for all $y\in B(x,2r)$, 
$$\avint_{z\in B(x,r)}\left|\EE(z-y) -\beta\right|\,dm(z)\lesssim1,$$
since $\EE(\cdot) = -c_1\,\log|\cdot|$ is in BMO. So, by the choice of $\wt B_Q$ and that $Q \in\DD^{db}_\omega$ we obtain
\begin{equation}\label{eq3**}
II\lesssim \frac{\omega^{x_B}(B(x,2r))}{r}\lesssim  \frac{\omega^{x_B}(100 B(Q))}{\mu(Q)} \lesssim \frac{\omega^{x_B}(Q)}{\mu(Q)}.
\end{equation}
Hence \rf{eq1**} follows from \rf{eq2**}, \rf{eqcv3'}
and \rf{eq3**}.

Choosing $\alpha=G(z,{x_B})$ with $z\in B(x,r)$ in \rf{eq1**} and averaging with respect Lebesgue measure for such $z$'s,  we get
\begin{align*}
\bigl|\wt\RR_r \omega^{x_B}(x)\bigr|& \lesssim \frac1{r^5}\!\iint_{B(x,r)\times B(x,r)} \!|G(y,{x_B}) - G(z,{x_B})|\,dm(y)\,dm(z) + \frac{\hm^{x_B}(B_0)}{\mu(B_0)}
+\frac{\omega^{x_B}(Q)}{\mu(Q)},
\end{align*}
where we understand that $G(z,{x_B})=0$ for $z\not\in\Omega$.
Now for $y,z\in B(x,r)$ and $\phi$ a radial smooth function such that $\phi\equiv 0$ in $B(0,2)$ and $\phi\equiv 1$ in $\R^{2}\setminus B(0,3)$ we write 
\begin{align*}
2\pi\,(G(y,{x_B}) - G(z,{x_B})) 
& = 
\log\frac{|z-{x_B}|}{|y-{x_B}|} - \int_{\pom} \log\frac{|z-\xi|}{|y-\xi|} \,d\omega^{x_B}(\xi) \\
& = 
\left(\log\frac{|z-{x_B}|}{|y-{x_B}|} - \int_{\partial\Omega} \phi\left(\frac{\xi-x}{r}\right)\,\log\frac{|z-\xi|}{|y-\xi|} \,d\omega^{x_B}(\xi)\right) \\
&\quad \!- \int_{\partial\Omega} \!\left(1-\phi\left(\frac{\xi-x}{r}\right)\right)\log\frac{|z-\xi|}{|y-\xi|} \,d\omega^{x_B}(\xi) 
= A_{y,z}  + B_{y,z}.
\end{align*}
Notice that the above identities also hold if $y,z\not\in\Omega$.
Let us observe that  
$$\frac{|z-{x_B}|}{|y-{x_B}|}\approx 1 \;\; \mbox{ and }\;\; 
\frac{|z-\xi|}{|y-\xi|}\approx 1\quad\mbox{ for $\xi\not \in B(x,2r)$.}
$$
We claim that  
\begin{equation}\label{eqlem33}
|A_{y,z}|\lesssim \frac{\omega^{x_B}(B(x,2{\delta_0}^{-1}r))}{\inf_{z\in B(x,2r)\cap \Omega} \omega^z(B(x,2{\delta_0}^{-1}r))}.
\end{equation}
We defer the details till the end of the proof.
Then, by Lemma \ref{lembourgain}, we get
$$\inf_{z\in B(x,2r)\cap \Omega}\omega^{z}(B(x,2{\delta_0}^{-1}r))\gtrsim \frac{\mu(B(x,2r))}{r}\geq \frac{\mu(\wt B_Q)}{r}.$$
and thus
$$\frac{|A_{y,z}|}r\lesssim \frac{\omega^{x_B}(B(x,2{\delta_0}^{-1}r))}{\mu(\wt B_Q)}\lesssim 
\frac{\omega^{x_B}(Q)}{\mu(Q)},$$
by the doubling properties of $Q$ (for $\omega^{x_B}$) and the choice of $\wt B_Q$.

To deal with the term $B_{y,z}$ we write:
\begin{align*}
|B_{y,z}|&\leq \int_{ B(x,3r)} \left(\left|\log\frac{r}{|y-\xi|}\right| +  \left|\log\frac{r}{|z-\xi|}\right|\right) \,d\omega^{x_B}(\xi).
\end{align*}
So we have
\begin{align*}
\iint_{B(x,r)\times B(x,r)} &|B_{y,z}|\,dm(y)\,dm(z) \lesssim 
r^2\int_{B(x,r)}\int_{ B(x,3r)} \left|\log\frac{r}{|y-\xi|}\right| \,d\omega^{x_B}(\xi)\,dm(y).
\end{align*}
Notice that for all $\xi\in B(x,3r)$,
$$\int_{B(x,r)}\left|\log\frac{r}{|y-\xi|}\right| \,dm(y)\lesssim r^2.$$
So by Fubini and $Q \in \DD^{db}_\omega$ we obtain
$$
\frac1{r^5}\iint_{B(x,r)\times B(x,r)} |B_{y,z}|\,dm(y)\,dm(z)\lesssim \frac{\omega^{x_B}(B(x,3r))}r\lesssim\frac{\omega^{x_B}(Q)}{\mu(Q)}.
$$
Together with the bound for the term $A_{y,z}$, this gives
$$\bigl|\wt\RR_r\omega^{x_B}(x)\bigr|\lesssim \frac{\omega^{x_B}(Q)}{\mu(Q)} + \frac{\omega^{x_B}(B_0)}{\mu(B_0)}\lesssim_A \frac{\omega^{x_B}(B_0)}{\mu(B_0)},$$
where the last inequality follows from the fact that $Q \in \good$.
\vv

It remains now to show \rf{eqlem33}. The argument uses ideas analogous to the ones for the proof of Lemma \ref{l:w>G} with some modifications. 
Recall that
\begin{align*}
A_{y,z} 
&= 
A_{y,z}({x_B}) = 
\log\frac{|z-{x_B}|}{|y-{x_B}|} - \int_{\partial\Omega} \phi\left(\frac{\xi-x}{r}\right)\,\log\frac{|z-\xi|}{|y-\xi|} \,d\omega^{x_B}(\xi)
\\
&=:
\log\frac{|z-{x_B}|}{|y-{x_B}|}-v_{x,y,z}({x_B})
\end{align*}
where $y,z\in B(x,r)$.
The two functions 
$$
q\longmapsto A_{y,z}(q)\qquad\text{ and }\qquad q\longmapsto \frac{c\,\omega^q(B(x,2{\delta_0}^{-1}r))}{\inf_{z\in B(x,2r)\cap \Omega} \omega_\Omega^{z}(B(x,2{\delta_0}^{-1}r))}$$
 are harmonic in $\Omega\setminus B(x,2r)$. Note that for all $q\in\partial B(x,2r)$
we clearly have
$$|A_{y,z}(q)|\leq c\leq \frac{c\,\omega^q(B(x,2{\delta_0}^{-1}r))}{\inf_{z\in B(x,2r)\cap \Omega} \omega_\Omega^{z}(B(x,2{\delta_0}^{-1}r))}.$$
 Since $A_{y,z}(q)=0$ for all $q\in\pom\setminus B(x,3r)$ except for a polar set we can apply the maximum principle 
 in \cite[Lemma 5.2.21]{Hel} 
 and obtain \rf{eqlem33}, as desired.
 
The case $Q \not\in \DD^{db}_\omega$ can be handled exactly as for the case of $n \geq 2$ and the proof is omitted.
\end{proof}

\vv
From the lemma above we deduce the following corollary.

\begin{lemma}\label{lemaxcor}
Let $\Omega$, $\mu$, $\eta$, $B$ and $B_0$ be as in the Main Lemma \ref{mainlemma} and Lemma \ref{lembola1}. 
Let 
$$\wt G_0 = G_0\setminus B(x_B,\eta\,r(B)).$$
For all $x\in\wt G_0$ we have
\begin{equation}\label{eqdk10}
\RR_*\omega^{x_B}(x) \lesssim \frac{\hm^{x_B}(B_0)}{\mu(B_0)},
\end{equation}
with the implicit constant depending on $n, A, \ve,\ve', \eta, \delta_0,\eta$.  
\end{lemma}

\begin{proof}
We need to show that for all $x\in\wt G_0$ and all $t>0$,
\begin{equation}\label{eqdk10t}
\bigl|\RR_t\omega^{x_B}(x)\bigr| \lesssim \frac{\hm^{x_B}(B_0)}{\mu(B_0)},
\end{equation}

Recall that the cubes from
$\DD_\omega$ are only defined for generations $k\geq 0$. However, by a suitable rescaling we can assume that they are defined
for $k\geq k_0$, where $k_0\in\Z$ can be arbitrary. So we suppose that there are cubes $Q\in\DD_\omega$ such that $\ell(Q)\geq r(B)$.

Denote by $\mathcal G_\eta$ the family of the cubes $Q \in \good$ such that $Q\cap \bigl(B_0 \setminus B(x_B,  \eta \, r(B))\bigr)\neq\varnothing$, $100B(Q)\subset B$, ${\delta_0} r(B_Q) \leq \eta\, r(B) $, and $Q \subset \partial \Omega \setminus B(x_B,  \frac{\eta}{2} \, r(B))$, so that \rf{eqdk0} holds for all $z\in Q\in\mathcal G_\eta$.

Given $x\in\wt G_0$, let $Q_x$ be the maximal cube from $\mathcal G_\eta$ that contains $x$. From the definition of
$\wt G_0$ and $\mathcal G_\eta$ it follows that such cube $Q_x$ exists and $\ell(Q_x)\approx r(B)\approx r(B_0)$, with the implicit constant depending on $\alpha$,
$\eta$, and $\delta_0$. Given $0<t\leq \ell(Q_x)$, let $P\in\DD_\omega$ be the cube containing $x$ such that $\ell(P)<t\leq \ell(
\wh P)$, where $\wh P$ stands for the parent of $P$. Note that $P,\wh{P}\in \mathcal G_\eta$, and by the Key Lemma \ref{lem:key}, we have
$$\bigl|\RR_{\ell(P)}\omega^{x_B}(x)\bigr| \lesssim \frac{\hm^{x_B}(B_0)}{\mu(B_0)}.$$
Then, taking also into account Lemma \ref{lemgrowth1}, we get
\begin{align*}
\bigl|\RR_t\omega^{x_B}(x)\bigr| & \leq \bigl|\RR_{\ell(P)}\omega^{x_B}(x)\bigr| + \frac{\omega^{x_B}(B(x,t))}{\ell(P)^n}\\
& \lesssim \frac{\hm^{x_B}(B_0)}{\mu(B_0)} + \frac{\omega^{x_B}(B(x,\ell(\wh P)))}{\ell(\wh P)^n}\lesssim \frac{\hm^{x_B}(B_0)}{\mu(B_0)}.
\end{align*}

In the case $t>\ell(Q_x)$, using that $\ell(Q_x)\approx r(B_0)$ together with a brutal estimate and Lemma \ref{lemaux1} we obtain
$$\bigl|\RR_t\omega^{x_B}(x)\bigr|\lesssim \frac{\|\omega^{x_B}\|}{\ell(Q_x)^n}\lesssim \frac1{r(B_0)^n}
\lesssim \frac{\hm^{x_B}(B_0)}{\mu(B_0)}.
$$
So the proof of \rf{eqdk10t} is concluded.
\end{proof}

\vv


\section{Proof of the Main Lemma \ref{mainlemma}}\label{sec8}

Recall that
$G_0= B_0  \setminus \bigcup_{Q\in \bad} Q$, and that in \rf{eq:bad-hm} and \rf{eq:bad-mu} we saw that
\begin{equation}\label{eqcc1**0}
\hm^{x_B}(G_0)\geq \ve'_2\, \hm^{x_B}(B_0 ),\qquad \mu(G_0)  \geq \ve'_1\,\mu(B_0 ).
\end{equation}
By Lemma \ref{lemgrowth2} is clear that there exists some constant $C_3$ such that
\begin{equation}\label{eqcc1**}
\omega^{x_B}(B(x,r))\leq C_3\, \dfrac{\hm^{x_B}(B_0)}{\mu(B_0)}\,r^n\quad\mbox{ for all $x\in G_0$ and all $r>0$.}
\end{equation}
Recall also that in Lemma \ref{lemaxcor} we introduced the set $\wt G_0 = G_0\setminus  B(x_B,\eta\,r(B))$ and we showed that
\begin{equation}\label{eqcc2}
\RR_*\omega^{x_B}(x) \lesssim \frac{\hm^{x_B}(B_0)}{\mu(B_0)}\quad\mbox{ for all $x\in\wt G_0$.}
\end{equation}

 We intend to apply the following T1 theorem:

\begin{theorem}\label{thm:T1}
Let $\nu$ be a compactly supported Borel measure in $\R^{d}$. Suppose that there  is an open set $H \subset \R^d$ with the following properties.
\begin{enumerate}[\quad(1)]
\item If $B_r$ is a ball of radius $r$ such that $\nu(B_r) > C_4r^n$, then $B_r \subset H$.
\item There holds  that $\int_{\R^n \setminus H} \RR_* \nu \,d\nu \le C_5\|\nu\|$.
\item $\nu(H) \le \delta_1\|\nu\|$, where $\delta_1 < 1$.
\end{enumerate}
Then there is a closed set $G$ satisfying that  $G \subset \R^d \setminus H$ and the following properties:
\begin{enumerate}[\quad (a)]
\item $\nu(G) \gtrsim \|\nu\|$.
\item $\nu(G \cap B_r) \le C_4r^n$ for every ball $B_r$ of radius $r$.
\item $\| 1_{G }\RR_\nu f \|_{L^2(\nu)} \lesssim \|f\|_{L^2(\nu)}$ for every $f \in L^2(\nu)$ such that $\supp f \subset G$.
\end{enumerate}
The implicit constants in (a) and (c) depend only on $n$, $d$, $C_4$, $C_5$, and $\delta_1$.
\end{theorem}

This result is a particular case of  the deep non-homogeneous Tb theorem of Nazarov, Treil and Volberg in \cite{NTV} (see also \cite{Volberg} and \cite[Theorem 8.14]{Tolsa-llibre}). 
\vv

 Set 
$$\nu:= \frac{\mu(B_0)}{\hm^{x_B}(B_0)}\, \hm^{x_B}|_{\alpha B_0} .$$
Observe that $\|\nu\|\approx\mu(B_0)$, because $\hm^{x_B}(\alpha B_0)\leq (1- \ve')^{-1} \hm^{x_B}(B_0)$. 
Also, by \rf{eqcc1**},
\begin{equation}\label{eqdk398}
\nu(B(x,r))\leq C_3\,r^n\quad\mbox{ for all $x\in G_0$ and all $r>0$.}
\end{equation}
From this fact, it easily follows that any ball $B_r$ such that 
$\nu(B_r)> 2^nC_3r^n$ does not intersect $G_0$. Indeed, if there exists $x\in G_0\cap B_r$, then
$$\nu(B(x,2r))\geq \nu(B_r)> C_3(2r)^n,$$
which contradicts \rf{eqdk398}.

For a fixed $0<\eta<1/10$ as in the statement of the Main Lemma \ref{mainlemma}, to simplify notation, we denote 
$$B_\eta = B(x_B,\eta\,r(B)).$$
There are two alternatives: either $\omega^{x_B}(B_\eta\cap G_0)> \frac{\ve_2'}2\,\omega^{x_B}(B_0)$ or 
$\omega^{x_B}(B_\eta\cap G_0)\leq \frac{\ve_2'}2\,\omega^{x_B}(B_0)$.
In the first case, from \rf{eqpoisson'}  we deduce that 
$$\mu(B_\eta\cap G_0) \geq \frac1A\,\omega^{x_B}(B_\eta\cap G_0)\,\frac{\mu(B_0)}{\omega^{x_B}(B_0)}
>  \frac{\ve_2'}{2A}\,\mu(B_0)\geq \frac{\ve_2'}{2C_2A}\,\mu(B)
,$$
by \rf{eq:Amu2001}.
So letting $\tau= \ve_2'/(2C_2A)$ (which does not depend on $\eta$), the alternative (i) of the Main Lemma \ref{mainlemma} holds.

In the second case, from \rf{eqcc1**0} we infer that
$$\omega^{x_B}(\wt G_0) = \omega^{x_B}(G_0) - \omega^{x_B}(B_\eta\cap G_0) \geq \ve_2'\,\omega^{x_B}(B_0)- \frac{\ve_2'}2\,\omega^{x_B}(B_0) = \frac{\ve_2'}2\,\omega^{x_B}(B_0).$$
We consider a closed set $\wt G_1\subset\wt G_0$ with 
$\omega^{x_B}(\wt G_1)\geq \frac{\ve_2'}3\,\omega^{x_B}(B_0)$,
which is equivalent to saying that
$\nu(\wt G_1)\geq \frac{\ve_2'}3\,\nu(B_0)$,
and we denote $H= \alpha\,B_0\setminus \wt G_1$. Because of the discussion just below \rf{eqdk398}, the assumption (1) of
the theorem holds with $C_4=2^nC_3$. Further, since $\nu(B_0)\approx \nu(\alpha B_0)$, we have
$$\nu(\wt G_1)\geq c\,\frac{\ve_2'}3\,\nu(\alpha B_0),$$
and thus
$$\nu(H) = \nu(\alpha B_0) - \nu( \wt G_1)
\leq \left(1-c\,\frac{\ve_2'}3\right)\,\nu(\alpha B_0) = \left(1-c\,\frac{\ve_2'}3\right)\,\|\nu\|,
$$
which ensures that the assumption (3) holds with $\delta_1= 1-c\,\frac{\ve_2'}3$.

To check that the assumption (2) is satisfied,
note that 
$$\nu= \frac{\mu(B_0)}{\hm^{x_B}(B_0)}\, \hm^{x_B} - \frac{\mu(B_0)}{\hm^{x_B}(B_0)}\, \hm^{x_B}|_{(\alpha B_0)^c},$$ 
and then it holds that
$$\RR_*\nu \leq \frac{\mu(B_0)}{\hm^{x_B}(B_0)}\, \RR_*\hm^{x_B} + \frac{\mu(B_0)}{\hm^{x_B}(B_0)}\,\RR_* (\hm^{x_B}|_{(\alpha B_0)^c}).$$
By \rf{eqcc2}, for any $x\in \alpha B_0\setminus H= \wt G_1$, the first term on the right hand side is uniformly bounded by some constant $C$.
On the other hand, using that $\wt G_1\subset B_0$ and taking into account Lemma \ref{lemaux1}, for the last term we have
$$ \frac{\mu(B_0)}{\hm^{x_B}(B_0)} \,\RR_*  ( \hm^{x_B}|_{(\alpha B_0)^c})(x) \lesssim_\alpha \frac{\mu(B_0)}{\hm^{x_B}(B_0)}  \,\frac{\hm^{x_B}((\alpha B_0)^c)}{r(B_0)^n} \lesssim \frac{\mu(B_0)}{\hm^{x_B}(B_0)}\, \frac1{r(B_0)^n} \lesssim 1.$$ 
So we get $\RR_*\nu(x) \lesssim 1$, for $\nu$-a.e. $x \in H^c$, which
yields (2) in Theorem \ref{thm:T1}.

We can now apply Theorem \ref{thm:T1} to obtain $G \subset \wt G_1\subset G_0\subset B_0$ such that 
\begin{enumerate}[(a)]
\item $\nu(G) \gtrsim \|\nu\| \approx \mu(B_0)\approx \mu(B)$.
\item $\nu(G \cap B_r) \le C_4r^n$ for every ball $B_r$ of radius $r$.
\item $\| 1_{G }\RR_\nu f \|_{L^2(\nu)} \lesssim \|f\|_{L^2(\nu)}$ for every $f \in L^2(\nu)$ satisfying that $\supp f \subset G$.
\end{enumerate}

Recall now that, by \rf{eqpoisson}, 
$$ k^{x_B} = \frac{d\omega^{x_B}}{d\mu} \approx \frac{\hm^{x_B}(B_0)}{\mu(B_0)} \quad\mbox{ in\; $G_0$}$$
and that $\nu =\frac{\mu(B_0)}{\hm^{x_B}(B_0)}\, k^{x_B}\,\mu|_{\alpha B_0}$.
First this implies that $\mu(G)\approx_{A,\ve_2'}\mu(B_0)$, and second, for any $f \in L^2(\mu)$ supported in $G$ it holds that
\begin{align*}
\int_{G} |\RR_\mu f |^2 \,d\mu &\approx  \int_{G} |\RR_\mu f |^2 \,d\nu\\
& = \int_{G} \left|\int K(x-y) f(y) (k^{x_B}(y))^{-1} \frac{\hm^{x_B}(B_0)}{\mu(B_0)}d\nu(y) \right|^2 \,d\nu(x)\\
& \lesssim \int_{G} \left| f(x) (k^{x_B}(x))^{-1} \frac{\hm^{x_B}(B_0)}{\mu(B_0)} \right|^2 \,d\nu(x)\\
&\approx  \int_{G} \left| f(x) \right|^2 \,d\mu(x).
\end{align*}
This concludes the proof of the Main Lemma \ref{mainlemma}. 

\vvv


\section{Proof of Theorem \ref{teo1}}\label{secend}\label{sec9}

In this section we will assume that $\Omega$ and $\mu$ satisfy the assumptions in Theorem \ref{teo1}.
For the proof we will need to work with the dyadic lattice of David-Mattila from Section \ref{sec:DMlatt} with 
the associated measure $\sigma=\mu$. This new dyadic lattice is now denoted by $\DD_\mu$. Recall that the cubes from
$\DD_\mu$ are only defined for generations $k\geq 0$. However, by a suitable rescaling we can assume that they are defined
for $k\geq k_0$, where $k_0\in\Z$ can be arbitrary.
\vv

\subsection{The Final Lemma and the
good $\lambda$ inequality}

Our next objective consists in proving the following.

\begin{lemma}[Final Lemma]\label{lemfinal}
For every  $R\in\DD^{db}_\mu$ there exists a subset $G_R\subset R$ with $\mu(G_R)\gtrsim\mu(R)$ such that
$\RR_{\mu|_{G_R}}:L^2(\mu|_{G_R})\to L^2(\mu|_{G_R})$ is bounded, with norm bounded above uniformly by some constant
depending on the various constants in the assumptions of Theorem \ref{teo1}.
\end{lemma}

\vv

Recall that by standard non-homogeneous Calder\'on-Zygmund theory, the boundedness of the operator
$\RR_{\mu|_{G_R}}:L^2(\mu|_{G_R})\to L^2(\mu|_{G_R})$ implies that $\RR_*$ is bounded from the space of finite real Radon measures $M(\R^{n+1})$ to $L^{1,\infty}(\mu)$.
See \cite[Chapter 2]{Tolsa-llibre}, for example.
Then, from Lemma \ref{lemfinal}, we deduce Theorem \ref{teo1} by means of the following result:

\begin{theorem}\label{thm:bigpiece}
Let $\mu$ be a Radon measure measure in $\R^{n+1}$ such that $\mu(B(x,r))\leq C\,r^n$ for all $r>0$. 
Suppose that the constant $C_0$ in the construction of $\DD_\mu$ in Lemma \ref{lemcubs} is big enough and let $\theta_0>0$. Suppose that for every cube $R\in\DD^{db}_\mu$ 
there exists a subset $G_R \subset R$ with $\mu(G_R) \geq \theta_0 \mu(R)$, such that $\RR_*$ is bounded from $M(\R^n)$ to $L^{1,\infty}(\mu|_{G_R})$, with norm bounded uniformly on $R$. Then $\RR_\mu$ is bounded in $L^p(\mu)$, for $1<p<\infty$, with its norm depending on $p$ and on the preceding constants.
\end{theorem}

This theorem is a variant of Theorem 2.22 from \cite{Tolsa-llibre}. In fact, in this reference the theorem is stated in terms of
``true" dyadic  cubes and it is proved by using a suitable good $\lambda$ inequality. Similar arguments, 
with minor variations,
work with cubes from the lattice $\DD_\mu$.
Below we just give a brief sketch of the proof, which highlights the modifications required with respect to 
Theorem 2.22 from \cite{Tolsa-llibre}.

\vv
\begin{proof}[Sketch of the proof of Theorem \ref{thm:bigpiece}]
Denote by $M_\mu$ the centered Hardy-Littlewood maximal operator:
$$M_\mu f(x) = \sup_{r>0} \frac1{\mu(B(x,r))} \int_{B(x,r)} |f|\,d\mu.$$

Arguing as in Theorem 2.22 from \cite{Tolsa-llibre}, it is enough to show that
for all $\ve>0$ there exists $\gamma=\gamma(\ve)>0$ such that for all $\lambda>0$,
\begin{equation}\label{eqgli0}
\mu\bigl(\bigl\{x\!:\RR_{\mu,*}f(x)>(1+\ve)\lambda,\, M_\mu f(x)\leq \gamma\lambda\bigr\}\bigr)
\leq \Bigl(1-\frac{\theta_0}{4}\Bigr)\,\mu\bigl(\bigl\{x\!:\RR_{\mu,*}f(x)>\lambda\bigr\}\bigr)
\end{equation}
for every compactly supported $f\in L^1(\mu)$. 

Denote 
$$\Omega_\lambda = \{x\!:\RR_{\mu,*}f(x)>\lambda\bigr\}.$$
The first step to prove \rf{eqgli0} consists in decomposing $\supp\mu \cap \Omega_\lambda$ into Whitney cubes
from the David-Mattila lattice $\DD_\mu$. Let us remark that in Theorem 2.22 from \cite{Tolsa-llibre},
the Whitney decomposition is performed in terms of ``true'' dyadic cubes from $\R^{n+1}$. 
The analogous result with the David-Mattila cubes is the following.

\begin{claim}\label{lemwhitney}
Assume that the cubes from $\DD_\mu$ are defined for the generations $k\geq k_0$, with $k_0\in\Z$ small enough.
Then there are cubes $Q_i\in\DD_\mu$ such that
$$\Omega_\lambda \cap \supp\mu= \bigcup_{i\in I} Q_i, $$
and so
that for some constants $T_0>10^4$ and $D_0\geq1$ the following holds:
\begin{itemize}
\item[(i)] $10^4B(Q_i) \subset \Omega$ for each $i\in I$.
\item[(ii)] $T_0 B(Q_i) \cap \Omega^{c} \neq \varnothing$ for each $i\in I$.
\item[(iii)] For each cube $Q_i$, there are at most $D_0$ cubes $Q_j$
such that $10^4B(Q_i) \cap 10^4B(Q_j) \neq \varnothing$. Further, for such cubes $Q_i$, $Q_j$, we have $\ell(Q_i)\approx
\ell(Q_j)$.
\item[(iv)] The family of doubling cubes 
$$\{Q_j\}_{j\in S}:= \{Q_i\}_{i\in I} \cap \DD_\mu^{db}$$
satisfies
\begin{equation} \label{bqht22}
\mu\biggl( \,\bigcup_{j\in S} Q_j \biggr) \geq \frac12\,
\mu(\Omega_\lambda),
\end{equation}
assuming the parameter $C_0$ in the construction of $\DD_\mu$ in Lemma \ref{lemcubs} big enough.
\end{itemize}
\end{claim}

Using the above decomposition, by arguments which are very similar to the ones in the proof of Theorem 2.22 from 
\cite{Tolsa-llibre}, one proves that for all $i\in I\cap S$,
$$\mu\bigl(\{x\in G_{Q_i}:\,\RR_{\mu,*}f(x)>(1+\ve)\lambda,\, M_\mu f(x)\leq \gamma\lambda\}\bigr)
\leq \frac{c\,\gamma}\ve\,\mu(Q_i),$$
and then one shows that this implies \rf{eqgli0} and the theorem follows.
\end{proof}
\vv

The arguments to prove the Claim \ref{lemwhitney} are quite similar to the ones for Lemma 2.23 of Theorem 2.22 from \cite{Tolsa-llibre}. However, the proof of the property (iv) is more tricky and so we show the details.

\vv
\begin{proof}[Proof of Claim \ref{lemwhitney}] 
Note that the open set $\Omega_\lambda$ is bounded (since $f\in L^1(\mu)$ is assumed to be compactly supported). So
assuming $k_0\in\Z$ to be sufficiently small (recall the comment at the beginning of Section \ref{secend}), the existence of cubes from $Q\in \DD_\mu$ with $\ell(Q)\approx \diam(\Omega_\lambda)$ is guarantied and so by standard arguments one can find cubes $Q_i\in\DD_\mu$ satisfying the properties (i) and (ii) above.
Indeed, the cubes $Q_i$, $i\in I$, can be defined as follows. Let $0<\delta_1<\frac1{100}$ be some small constant to be fixed below. Then, for all
$x\in\supp\mu\cap\Omega_\lambda$, let $Q_x\in\DD_\mu$ be the maximal cube containing $x$ such that
\begin{equation}\label{eqdelta0}
\ell(Q_x)\leq \delta_1\,\dist(x,\partial\Omega_\lambda).
\end{equation}
Let $\{Q_i\}_{i\in I}$ be the subfamily of the maximal and thus disjoint cubes from $\{Q_x\}_{x\in\supp\mu\cap\Omega_\lambda}$.
The properties (i) and (ii) are immediate (assuming $\delta_1$ small enough). On the other hand, (iii) follows easily from the following:
\begin{itemize}
\item[(iii')] If $10^4B(Q_i) \cap 10^4B(Q_j) \neq \varnothing$ for some $i,j\in I$, then $|J(Q_i)-J(Q_j)|\leq 1$, assuming
$\delta_1$ small enough in \rf{eqdelta0} (here $J(Q_i)$ and $J(Q_j)$ are the generations to which $Q_i$ and $Q_j$ belong,
respectively).
\end{itemize}
To prove this, take $i,j\in I$ as above. By definition, there exists some point $p_i\in Q_i$ such that
$\ell(Q_i)\leq \delta_1\,\dist(p_i,\partial\Omega_\lambda)$. So for any $p_j\in Q_j$, by the triangle inequality
$$\ell(Q_i)\leq \delta_1\,\bigl(|p_i-p_j| + \dist(p_j,\partial\Omega_\lambda)\bigr).$$
From the condition $10^4B(Q_i) \cap 10^4B(Q_j)\neq \varnothing$, we get $|p_i-p_j|\leq C(A_0,C_0)\bigl(\ell(Q_i) +\ell(Q_j)
\bigr)$ and thus
$$\ell(Q_i)\leq \delta_1\,C(A_0,C_0)\bigl(\ell(Q_i) +\ell(Q_j)\bigr) +  \delta_1\, \dist(p_j,\partial\Omega_\lambda)\bigr).$$
On the other hand, from the definition of $\ell(Q_j)$ we infer that the parent $\wh Q_j$ of $Q_j$ satisfies
$$A_0\,\ell(Q_j) = \ell(\wh Q_j)>\delta_1\,\dist(p_j,\partial\Omega_\lambda).$$
So we derive
$$\ell(Q_i)\leq \delta_1\,C(A_0,C_0)\bigl(\ell(Q_i) +\ell(Q_j)\bigr) +  A_0\,\ell(Q_j).$$
Taking $\delta_1$ small enough (depending on $A_0$ and $C_0$), this implies that
$$\ell(Q_i)\leq 2  A_0\,\ell(Q_j).$$
Since the side-lengths of cubes from $\DD_\mu$ are of the form $56 C_0A_0^k$, $k\in\Z$, and $A_0\gg 2$, the above estimate is equivalent to saying that $\ell(Q_i)\leq   A_0\,\ell(Q_j)$.
By analogous arguments, it follows that $\ell(Q_j)\geq   A_0\,\ell(Q_j)$, and so (iii') is proved.

\vv
Finally, we show that the property (iv) holds.
 If $Q_i\in I\setminus S$, then 
$$\mu(Q_i)\leq \mu(100B(Q_i))\leq \frac1{C_0}\,\mu(10^4B(Q_i)),$$
by \rf{eqdob23}, assuming $C_0>100$. Then we deduce
\begin{equation}\label{eq**23}
\sum_{i\in I\setminus S} \mu(Q_i) \leq \frac1{C_0} \sum_{i\in I\setminus S} \mu(10^4B(Q_i)).
\end{equation}
To bound the last sum we need to estimate the number of cubes $Q_i$, $i\in I\setminus S$, such that $x\in 10^4B(Q_i)$,
for a given $x\in\supp\mu$. From the property (iii') it is clear that such cubes can belong at most to two different generations.
Since the cubes $Q_i$, $i\in I\setminus S$, are not from $\DD_\mu^{db}$, by construction we have $r(B(Q_i))=A_0^{-J(Q_i)}$.
So all the cubes $Q_i$ of a given generation $J_0$ such that $x\in 10^4B(Q_i)$ are contained $B(x,2\cdot10^4\,A_0^{-J_0})$.
Since the balls $B(Q_i)$ of a fixed generation $J_0$ are disjoint, arguing with Lebesgue measure, we have
\begin{align*}
A_0^{-J_0(n+1)}\,\#\bigl\{i\in I\setminus S: x\in 10^4B(Q_i) \text{ and } J(Q_i)=J_0\bigr\} & =\!\!
\sum_{\substack{i\in I\setminus S: 
x\in 10^4B(Q_i)\\ J(Q_i)=J_0}} \!\!r(B(Q_i))^{n+1}\\
& \leq (2\cdot 10^4\,A_0^{-J_0})^{n+1}.
\end{align*}
Using this estimate and the fact there are at most two possible values for $J_0$, we get
$$\#\bigl\{i\in I\setminus S: x\in 10^4B(Q_i)\bigr\} \leq 2 \,(2\cdot 10^4)^{n+1}.$$
The key point of this estimate is that the value on the right hand side is an absolute constant that does not depend on the parameters $C_0$ and $A_0$ from the construction of the lattice $\DD_\mu$ in Lemma \ref{lemcubs}. Then, plugging this 
inequality into \rf{eq**23} and using also (i) we deduce
$$\sum_{i\in I\setminus S} \mu(Q_i) \leq \frac1{C_0} \int_{\Omega_\lambda} \sum_{i\in I\setminus S} \chi_{10^4B(Q_i)}(x)\,d\mu(x)
\leq \frac{2 \,(2\cdot 10^4)^{n+1}}{C_0}\,\mu(\Omega_\lambda)\leq \frac12\,\mu(\Omega_\lambda),$$
assuming that the parameter $C_0$ is chosen big enough in Lemma \ref{lemcubs} for the last inequality. This
yields
$$\mu\biggl( \,\bigcup_{j\in S} Q_j \biggr) \geq \mu(\Omega_\lambda) - \sum_{j\in I\setminus S} \mu(Q_j)\geq
\frac12\,
\mu(\Omega_\lambda),$$
as wished and concludes the proof of \rf{bqht22}.
\end{proof}

The next Subsections \ref{sub92}-\ref{subfi} are devoted to the proof of the Final Lemma \ref{lemfinal}.
\vv


\subsection{The nice and the ugly cubes}\label{sub92}

Given $Q\in\DD_\mu^{db}$, for $\lambda>0$, denote
$$Q_\lambda = \bigl\{x\in Q:\dist(x,\supp\mu\setminus Q)\geq \lambda\,\ell(Q)\bigr\}.$$
Recall that, by the thin boundary property \rf{eqfk490} and the fact that $Q$ is doubling,
$$
\mu\bigl(Q\setminus Q_\lambda\bigr)  
\leq
c\,\lambda^{1/2}\,\mu(3.5B_Q)\leq c'\,\lambda^{1/2}\,\mu(Q).$$
Thus, for $\lambda_0>0$ small enough,
$$\mu\bigl(Q_{\lambda_0}\bigr) \geq \frac12\,\mu(Q).$$
Now consider an open ball $B'$ whose center lies in $Q_{\lambda_0}$, with $r(B')= \dfrac{{\delta_0}\,\lambda_0}{10}\,\ell(Q)$,
such that $\mu(B')$ is maximal among such balls, and so
$$\mu(B')\geq C({\delta_0},\lambda_0)\,\mu(Q_{\lambda_0})\gtrsim \mu(Q).$$
Suppose that the constant $C_1$ in the definition of balls with thin boundaries in \rf{eqthin4} has been
chosen big enough. Then there is another ball $B$, concentric with $B'$, with $C_1$-thin boundary, and such that
$2{\delta_0}^{-1}B'\subset B\subset 2.2 {\delta_0}^{-1}B'$. 
For the proof, with cubes instead of balls, we refer the reader to Lemma 9.43 of \cite{Tolsa-llibre}, for example. 
Observe now that $B$ satisfies the assumptions of Main Lemma \ref{mainlemma}, assuming $C_2$ big enough.
Indeed, since 
\begin{equation}\label{eqah31}
2B\cap\supp\mu\subset 4.4 {\delta_0}^{-1}B'\cap\supp\mu\subset Q\quad\text{ and }\quad B'\subset \frac{\delta_0}2 B,
\end{equation}
we get
$$\mu(2B)\leq \mu(4.4 {\delta_0}^{-1}B')\leq\mu(Q)\leq  C_2({\delta_0},\lambda_0)\,\mu(B')  \leq  C_2({\delta_0},\lambda_0)\,\mu(\tfrac{\delta_0}2 B).$$
Notice that $C_2=C_2({\delta_0},\lambda_0)$ is an absolute constant which depends on $n$, but not on other parameters such as
the parameters $\ve$ and $\ve'$ in Theorem \ref{teo1}.
The existence of a point $x_B$ as in the Main Lemma such that \rf{eq:Amu10} holds is guarantied by the assumptions
of Theorem \ref{teo1} applied to $B$, with $c_{db} =  C_2({\delta_0},\lambda_0)$.

Let $\eta\in(0,1/10)$ some small constant whose precise value will be chosen below, depending on $\tau,\delta_0,\lambda_0$  (note that the constant $\tau$ from the Main Lemma is independent of $\eta$).
By the Main Lemma, one of the following statement holds:
\begin{itemize}
\item[(i)] Either
$$\mu(B(x_B,\eta\,r(B)))\geq \tau\,\mu(B),$$
where $\tau$ is some positive constant depending on  $C_\mu$, $\ve$, $\ve'$, $C_1$ and $C_2$ (but not on $\eta$); 
or
\vv

\item[(ii)] there exists some subset $G_B\subset B$ with $\mu(G_B)\geq \theta\mu(B)$, $\theta>0$, such that the
 Riesz transform $\RR_{\mu|_{G_B}}: L^2(\mu|_{G_B}) \to L^2(\mu|_{G_B})$ is bounded. The constant $\theta$ and the 
 $L^2(\mu|_{G_B})$ norm depend only on $C_\mu$, $\ve$, $\ve'$, $C_1$, $C_2$, and $\eta$.
\end{itemize} 

\vv
If (ii) holds, we say that $Q$ is {\em nice}, and we write $Q\in\NN$. Otherwise,  i.e., in case (i), we say that $Q$ is {\em ugly} and we
write $Q\in\UU$. Clearly, since $2B\cap\supp\mu\subset Q$ (by \rf{eqah31}), we have:
\begin{itemize}
\item If $Q\in \DD^{db}_\mu\cap\NN$, then there exists $\wt G_Q\equiv G_B\subset Q$ such that
\begin{equation}\label{eq*90}
\mu(\wt G_Q)\approx\mu(Q)\quad \mbox{ and }\quad
\RR_{\mu|_{\wt G_Q}}\!: L^2(\mu|_{\wt G_Q}) \to L^2(\mu|_{\wt G_Q})\,\mbox{ is bounded,}
\end{equation}
with the implicit constants in both estimates uniform on $Q$. Further, 
\begin{equation}\label{eqdist*1}
\dist(\wt G_Q,\,\supp\mu\setminus Q) \geq r(B)\gtrsim \ell(Q).
\end{equation}

\vv

\item If $Q\in \DD^{db}_\mu\cap\UU$, then 
\begin{equation}\label{eq*91}
\mu(B(x_B,\eta\,r(B)))\geq \tau\,C({\delta_0},\lambda_0)\mu(B).
\end{equation}

\end{itemize}

Note that since $x_B\in\frac{\delta_0}2B$, we have
$$\supp\mu\cap B(x_B,\eta\,r(B))\subset \supp\mu\cap B\subset Q.$$
Assuming $Q\in\DD^{db}_\mu\cap \UU$, since $B(x_B,\eta\,r(B))$ is covered by a bounded number of cubes of side length comparable to $\eta\,r(B)$, we infer that
there exists a cube $\wt P_Q\subset Q$ which satisfies:
\begin{equation}\label{eq*1}
\ell(\wt P_Q)\approx\eta\,r(B)\approx C(\delta_0,\lambda_0)\,\eta\,\ell(Q),
\end{equation}
\begin{equation}\label{eq*2}
\mu(\wt P_Q)\geq C({\delta_0},\lambda_0,\tau)\,\mu(Q),
\end{equation}
and
\begin{equation}\label{eq*3}
\Theta_\mu(\wt P_Q)\geq \frac{C({\delta_0},\lambda_0,\tau)}{\eta^{n}}\,\Theta_\mu(Q).
\end{equation}
Consider now the smallest doubling cube $P_Q\in\DD_\mu^{db}$ such that $\wt P_Q\subset P_Q\subset Q$.
Clearly, $P_Q\subset Q$ and the estimates \rf{eq*1} and \rf{eq*2} also hold with $\wt P_Q$ replaced by $P_Q$.
It also easy to see that \rf{eq*3} is satisfied:

\begin{claim}
Assume $Q\in\DD^{db}_\mu\cap\UU$.
Then
$$\Theta_\mu(P_Q)\geq C^{-1}\,\Theta_\mu(\wt P_Q)\geq \frac{C({\delta_0},\lambda_0,\tau)}{\eta^{n}}\,\Theta_\mu(Q).$$
\end{claim}

\begin{proof}
Indeed, by Lemma \ref{lemcad23}, since all the
 intermediate cubes $S$ with
$\wt P_Q\subsetneq S\subsetneq P_Q$ are non-doubling, 
we have
$$\Theta_\mu(\wt P_Q)\lesssim\Theta_\mu(100B(\wt P_Q))\leq C_0\,A_0^{-9n(J(\wt P_Q)-J(P_Q)-1)}\,\Theta_\mu(100B(P_Q))
\lesssim\,\Theta_\mu(P_Q)
,$$
since $J(\wt P_Q)-J(P_Q)\geq0$ and $\Theta_\mu(100B(P_Q))\approx\Theta_\mu(P_Q)$, because $P_Q\in\DD^{db}_\mu$.
\end{proof}

Note that for $Q\in\DD^{db}_\mu\cap\UU$, from the estimates \rf{eq*2} and \rf{eq*3} applied to $P_Q$, we deduce that
\begin{equation}\label{equg1}
\Theta_\mu(P_Q)\,\mu(P_Q)\geq \frac{C(\tau,{\delta_0},\lambda_0)}{\eta^{n}}\,\Theta_\mu(Q)\,\mu(Q)
\gg \Theta_\mu(Q)\,\mu(Q),
\end{equation}
assuming $\eta$ small enough.
\vv


\subsection{The corona decomposition}

In order to prove the Final Lemma \ref{lemfinal} we have to show that for any 
 $R\in\DD^{db}_\mu$ there exists a subset $G_R\subset R$ with $\mu(G_R)\approx\mu(R)$ such that
$\RR_{\mu|_{G_R}}:L^2(\mu|_{G_R})\to L^2(\mu|_{G_R})$ is bounded uniformly on $R$.
If $R\in\NN$, then we take $G_R=\wt G_R$ and we are done. For a general cube $R\in\DD_\mu^{db}$, in order to find an
appropriate set $G_R$ we have to 
construct a corona decomposition of $\mu|_R$.

 For every $Q\in\DD_\mu^{db}(R)$  we define a family of stopping cubes $\sss(Q)\subset\DD_\mu$ as follows:
\begin{itemize}
\item[(a)] If $Q\in \NN$, then we set $\sss(Q)=\varnothing$.
\vv

\item[(b)] If $Q\in \UU$, then $\sss(Q)$ consists of all the cubes from $\DD_\mu$
which are contained in $Q$ and are of the same generation as the cube $P_Q$ defined in Subsection \ref{sub92}.
\end{itemize}

Given a cube $P\in\DD_\mu$, we denote by $\MD(P)$ the family  of maximal cubes 
(with respect to inclusion) from $\DD_\mu^{db}(P)$. Recall that, by Lemma \ref{lemcobdob}, this family covers $\mu$-almost
all $P$. Moreover, by Lemma \ref{lemcad23} it follows that if $S\in\MD(P)$, then 
$$\Theta_\mu(2B_S)\leq c\,\Theta_\mu(2B_P).$$

Given $Q\in\DD^{db}_\mu$, we denote
$$\nex(Q) = \bigcup_{P\in\sss(Q)} \MD(P).$$
So if $Q\in\NN$, then $\nex(Q)=\varnothing$. On the other hand,
if $Q\in\UU$, then $P_Q\in\nex(Q)$, and thus by \rf{equg1}, if $\eta$ is chosen small enough
in the Main Lemma \ref{mainlemma},
\begin{equation}\label{eq*10}
\sum_{P\in\nex(Q)}\Theta_\mu(P)\,\mu(P)\geq \Theta_\mu(P_Q)\,\mu(P_Q)
\geq 2\,\Theta_\mu(Q)\,\mu(Q).
\end{equation}


We are now ready to construct the family of the $\ttt$ cubes of the corona construction. We will have
$\ttt=\bigcup_{k\geq0}\ttt_k$. First we set
$$\ttt_0=\{R\}.$$
Assuming that $\ttt_k$ has been defined, we set
$$\ttt_{k+1}  = \bigcup_{P\in\ttt_k} \nex(P).$$
Note that the families $\nex(Q)$, with $Q\in\ttt_k$, are pairwise disjoint. 
Observe also that $\ttt\subset \DD_\mu^{db}(R)$.

\vv


\subsection{The packing condition}

Next we prove a key estimate.
\vv

\begin{claim}\label{claimkey}
If $\eta$ is chosen small enough (so that \rf{eq*10} holds for $Q\in\UU$), then 
\begin{equation}\label{eqsum441}
\sum_{Q\in\ttt}\Theta_\mu(Q)\,\mu(Q)\leq C\,\mu(R).
\end{equation}
\end{claim}

\begin{proof}
For a given $k\geq 0$, we denote 
$$\ttt_0^k = \bigcup_{0\leq j\leq k}\ttt_j,$$
and also
$$\NN_0^k = \NN \cap \ttt_0^k\quad \mbox{ and }\quad \UU_0^k = \UU \cap \ttt_0^k.$$

To prove \rf{eqsum441}, first we deal with the cubes from the family $\UU$.
Recall that, by \rf{eq*10},  the cubes $Q$ from this family satisfy
$$
\sum_{P\in\nex(Q)}\Theta_\mu(P)\,\mu(P)\geq 2\,\Theta_\mu(Q)\,\mu(Q),
$$
and thus
\begin{align*}
\sum_{Q\in \UU_0^k} \!\!\Theta_\mu(Q)\,\mu(Q) & \leq \frac1{2} \sum_{S\in \UU_0^k}\sum_{Q\in \nex(S)}\!\Theta_\mu(Q)\,\mu(Q)\leq \frac1{2}\sum_{Q\in \ttt_0^{k+1}}\Theta_\mu(Q)\,\mu(Q),
\end{align*}
because the cubes from $\nex(Q)$ with $Q\in\ttt_0^k$ belong to $\ttt_0^{k+1}$. 
So we have
\begin{align*}
\sum_{Q\in \ttt_0^k}\! \Theta_\mu(Q)\,\mu(Q)  &= \sum_{Q\in \NN_0^k} \Theta_\mu(Q)\,\mu(Q)+
\sum_{Q\in \UU_0^k} \Theta_\mu(Q)\,\mu(Q)\\
& \leq \sum_{Q\in \NN_0^k} \Theta_\mu(Q)\,\mu(Q) +
\frac1{2}\sum_{Q\in \ttt_0^{k}}\!\Theta_\mu(Q)\,\mu(Q) +
 c\,C_\mu\,\mu(R),
\end{align*}
where we took into account that $\Theta_\mu(Q)\lesssim C_\mu$ for every $Q\in\ttt$ (and in particular for all $Q\in\ttt_{k+1}$)
for the last inequality.
So we deduce that
$$\sum_{Q\in \ttt_0^k} \Theta_\mu(Q)\,\mu(Q)  \leq 2\sum_{Q\in \NN_0^k} \Theta_\mu(Q)\,\mu(Q) +
 c \,C_\mu\,\mu(R).$$
Letting $k\to\infty$, we derive
\begin{equation}\label{eqaaii}
\sum_{Q\in \ttt} \Theta_\mu(Q)\,\mu(Q)\leq 2\sum_{Q\in \ttt\cap\NN} \Theta_\mu(Q)\,\mu(Q)+  c \,C_\mu\,\mu(R).
\end{equation}
Now notice that
$$ \sum_{Q\in \ttt\cap \NN} \Theta_\mu(Q)\,\mu(Q)\leq c\,C_\mu\,
\mu(R),$$
using the polynomial growth of $\mu$ and that the nice cubes $Q\in\ttt\cap\NN$ are pairwise disjoint, since $\nex(Q)=\varnothing$ for such cubes $Q$, by construction.
\end{proof}

\vv


\subsection{The measure \boldmath$\nu$ and the \boldmath $L^1(\nu)$ norm of $\RR_*\nu$}\label{subfi}

Recall that in \rf{eq*90} we have introduced the good sets $\wt G_Q$ for the nice cubes $Q\in\NN$.
In particular, $\wt G_R$ has already been defined in the case $R\in\NN$. When $R\in\UU$ we set 
$$\wt G_R = \biggl(R\setminus \bigcup_{Q\in \NN} Q\biggr) \cup \bigcup_{Q\in \NN}\wt G_Q.$$
Note that this identity is also valid if $R\in\NN$.
Since $\mu(\wt G_Q)\approx \mu(Q)$ for every $Q\in\NN$, we deduce that
$$\mu(\wt G_R)\approx \mu(R).$$

Denote $\nu = \mu|_{\wt G_R}$. To complete the proof of Lemma \ref{lemfinal}, we wish to show that there exists $G_R\subset\wt G_R$ with $\nu(G_R)\approx\nu(\wt G_R)$ such that
$\RR_{\nu|_{G_R}}:L^2(\nu|_{G_R}) \to L^2(\nu|_{G_R})$ is bounded. The main step is the following.




\begin{claim}
We have
$$\|\RR_*\nu\|_{L^1(\nu)}\leq C\,\nu(R).$$
\end{claim}

\begin{proof}
Given $Q\in\ttt$ and $x\in Q$, we denote by $r(x,Q)$ the radius of the ball $B(P)$ with $P\in\nex(Q)$ such that $x\in P$.
If such cube $P$ does not exist (for example, because $Q\in\NN$), we set $r(x,Q)=0$.

Given $0<\ve_1\leq\ve_2$, we use the double cut-off Riesz transform defined by
$$\RR_{\ve_1,\ve_2}\nu(x) = \RR_{\ve_1}\nu(x) - \RR_{\ve_2}\nu(x).$$
For $x\in R$, we set
\begin{align}\label{eqlab11}
\RR_*\nu(x) &\leq \sup_{\ve>r(B(R))}|\RR_\ve\nu(x)| + \sum_{Q\in\ttt\cap\UU}\chi_Q(x)\,\sup_{r(B(Q))\geq\ve>r(x,Q)}|\RR_{\ve,r(B(Q))}\nu(x)|\\
&\quad + \sum_{Q\in\ttt\cap\NN}\chi_Q(x)\,\sup_{r(B(Q))\geq\ve>0}|\RR_{\ve,r(B(Q))}\nu(x)|.\nonumber
\end{align}
Observe first that
$$\sup_{\ve>r(B(R))}|\RR_\ve\nu(x)|\leq \frac{\|\nu\|}{r(B(R))}
\lesssim \Theta_\nu(R)\leq \Theta_\mu(R)\lesssim C_\mu.$$
On the other hand, for $x\in Q\in\ttt\cap\NN$, we write
$$
\sup_{r(B(Q))\geq\ve>0}|\RR_{\ve,r(B(Q))}\nu(x)|\lesssim \RR_*(\nu|_{100B(Q)})(x).
$$

Finally, consider case $x\in Q\in\ttt\cap\UU$. Let $P_x\in\nex(Q)$ be such that $P_x\ni x$ (with $P_x=\varnothing$ is $P_x$ does not exist). Then we have
\begin{align*}
\sup_{r(B(Q))\geq\ve>r(x,Q)}|\RR_{\ve,r(B(Q))}\nu(x)|& \lesssim \sum_{S\in\DD_\mu:Q\supset S\supset P_x}\Theta_\nu(100B(S))\\
&\leq  \sum_{S\in\DD_\mu:Q\supset S\supset P_x}\Theta_\mu(100B(S)).
\end{align*}
Recall now the way that the cube $P_x\in\nex(Q)$ has been constructed: there exists some cube $\wt P_x\in\sss(Q)$
such that $\ell(\wt P_x)\approx\ell(Q)$ and $P_x$ is the maximal cube from $\DD_\mu^{db}(\wt P_x)$ that contains $x$.
Then by Lemma \ref{lemcad23}, 
$$\sum_{S\in\DD_\mu:\wt P_x\supset S\supset P_x}\Theta_\mu(100B(S))\lesssim \Theta_\mu(100B(\wt P_x))\lesssim \Theta_\mu(100B(Q)),$$
taking into account for the last inequality that $100B(\wt P_x)\subset 100B(Q)$ and that $r(B(\wt P_x))\approx r(B(Q))$.
This trivial estimate also yields 
$$\sum_{S\in\DD_\mu:Q\supset S\supset \wt P_x}\Theta_\mu(100B(S))\lesssim  \Theta_\mu(100B(Q)).$$
So we deduce that, for $x\in Q\in\ttt\cap\UU$,
$$\sup_{r(B(Q))\geq\ve>r(x,Q)}|\RR_{\ve,r(B(Q))}\nu(x)|\lesssim \Theta_\mu(100B(Q))\lesssim \Theta_\mu(Q),$$
using also that $Q\in\DD_\mu^{db}$ for the last inequality.

From \rf{eqlab11} and the above estimates, we infer that
$$\RR_*\nu(x) \lesssim \Theta_\mu(R) + \sum_{Q\in\ttt\cap\UU}\chi_Q(x)\,\Theta_\mu(Q) 
+ \sum_{Q\in\ttt\cap\NN}\chi_Q(x)\,\RR_*(\nu|_{100B(Q)})(x).$$
Integrating on $R$ with respect to $\nu$, we get
\begin{align}\label{eqlab14}
\|\RR_*\nu\|_{L^1(\nu)} &\lesssim \Theta_\mu(R)\,\nu(R) + \sum_{Q\in\ttt\cap\UU}\Theta_\mu(Q)\,\nu(Q) 
+ \sum_{Q\in\ttt\cap\NN}\int_Q\RR_*(\nu|_{100B(Q)})\,d\nu\\
& \lesssim \sum_{Q\in\ttt}\Theta_\mu(Q)\,\mu(Q) 
+ \sum_{Q\in\ttt\cap\NN} \|\RR_*(\nu|_{100B(Q)})\|_{L^1(\nu|_Q)},\nonumber
\end{align}
where we took into account that $R\in\ttt$ in the last inequality.
By \rf{eqsum441} we know that the first sum on the right hand side does not exceed $C\,\mu(R)$. To deal with the last sum,
recall first that, by \rf{eqdist*1},
$$\dist(Q\cap \supp\nu,\supp\nu\setminus Q) \geq \dist(\wt G_Q,\,\supp\mu\setminus Q) \gtrsim \ell(Q).$$
Thus, for all $x\in Q\cap \supp\nu$, 
\begin{align*}
\RR_*(\nu|_{100B(Q)})(x)& \leq \RR_*(\nu|_{100B(Q)\setminus Q})(x) + \RR_*(\nu|_Q)(x) \\
& \lesssim \Theta_\nu(100B(Q)) + \RR_*(\nu|_Q)(x) \lesssim \Theta_\mu(Q) + \RR_*(\nu|_Q)(x).
\end{align*}
By the Cauchy-Schwarz inequality we obtain
$$\|\RR_*(\nu|_{100B(Q)}\|_{L^1(\nu|_Q)}\leq \Theta_\mu(Q) \,\nu(Q) +\|\RR_*(\nu|_{Q})\|_{L^2(\nu|_Q)}\,\nu(Q)^{1/2}.$$
Since $\RR_{\mu|_{\wt G_Q}}$ is bounded in $L^2(\mu|_{\wt G_Q})$, by standard non-homogeneous Calder\'on-Zygmund theory,
it follows that $\RR_{\mu|_{\wt G_Q},*}$ is bounded in $L^2(\mu|_{\wt G_Q})$, and thus
$$\|\RR_*(\nu|_{Q})\|_{L^2(\nu|_Q)} = \|\RR_*(\mu|_{\wt G_Q})\|_{L^2(\mu|_{\wt G_Q})} \lesssim\mu(\wt G_Q)^{1/2} = \nu(Q)^{1/2}.$$
Therefore,
$$\|\RR_*(\nu|_{100B(Q)}\|_{L^1(\nu|_Q)}\leq \Theta_\mu(Q) \,\nu(Q) + \nu(Q) \lesssim \mu(Q).$$
Since the cubes from $\ttt\cap\NN$ are pairwise disjoint, from \rf{eqlab14} we deduce that
$$\|\RR_*\nu\|_{L^1(\nu)} \lesssim \mu(R) + \sum_{Q\in\ttt\cap\NN}\mu(Q) \lesssim \mu(R)\approx \nu(R).$$
\end{proof}

\subsection{Proof of Lemma \ref{lemfinal}}
To find the set $G_R\subset R$ with $\mu(G_R)\gtrsim\mu(R)$ such that
$\RR_{\mu|_{G_R}}:L^2(\mu|_{G_R})\to L^2(\mu|_{G_R})$ is bounded (with norm independent of $R$) we just have to apply
Theorem \ref{thm:T1} to the measure $\nu$, with $H=\varnothing$, and take into account that 
$$\|\RR_*\nu\|_{L^1(\nu)}\lesssim \|\nu\|$$ and that $\|\nu\|=\nu(R)\approx \mu(R)$. This completes the proof of Lemma
\ref{lemfinal}, and hence of Theorem \ref{teo1}.
\fiproof

\vv


\section{Harmonic measure in uniform domains}\label{sec10}

First, in this section we will
prove some general estimates involving harmonic measure and Green's function on uniform domains. In particular, we will prove
Theorem \ref{teounif}. Finally we will show how Theorem \ref{teo2} follows from Theorem \ref{teo1}
and Theorem \ref{teounif}.

Let $\Omega\subset \R^{n+1}$ be a uniform domain and let $x_0\in\Omega$. Let $d(x_0)=\dist(x_0,\Omega)$.
 In the case $n\geq2$, it is easy to check
that for all $y\in\partial B(x,d(x_0)/4)$,
\begin{equation}\label{eqdhgj2}
G(x_0,y)\approx \frac1{d(x_0)^{n-1}}.
\end{equation}
In the case $n=1$, we have
\begin{equation}\label{eqdhgj3}
G(x_0,y)\gtrsim 1.
\end{equation}
However, as far as we know, the converse inequality is not guarantied. 
On the other hand, by a Harnack chain argument it is easy to check
that $G(x_0,y)\approx G(x_0,y')$ for all $y,y'\in \partial B(x_0,d(x_0)/4)$, where the implicit constant is an absolute constant.

For any $n\ge1$, for a given $x_0\in\Omega$, we define
$$\rho(x_0) = \avint_{\partial B(x_0,d(x_0)/4)}G(x_0,y)\,d\HH^n(y),$$x
so that $G(x_0,y)\approx \rho(x_0)$ for all $y\in \partial B(x_0,d(x_0)/4)$. In the case $n\geq 2$, by \rf{eqdhgj2} we have
$\rho(x_0)\approx d(x_0)^{1-n}$, and in the case $n=1$, by \rf{eqdhgj3} it follows just that $\rho(x_0)\gtrsim1$.

\begin{lemma}\label{l:w<G}
Let $n\geq1$ and let $\Omega\subsetneq \R^{n+1}$ be a uniform domain  and $B$ a ball centered at $\partial\Omega$ with radius $r$.  Suppose that there exists a point $x_B \in \Omega$ so that the ball $B_0:=B(x_{B},r/C)$ satisfies $4B_0\subset \Omega\cap B$ for some $C>1$. Then, for  $0<r\leq r_\Omega$ (where $r_\Omega$ is some constant sufficiently small),
and $\tau>0$,
\begin{equation}\label{eq:Green-upperbound}
 \omega^{x}(B)\approx \omega^{x_{B}}(B)\, \rho(x_B)^{-1}\, G(x,x_{B})\,\, \,\,\,\,\,\,\text{for all}\,\, x\in \Omega\backslash  (1+\tau)B.
 \end{equation}
The implicit constant in \rf{eq:Green-upperbound} 
depends only $C$, $\tau$,  $n$, and the uniform character of $\Omega$. The constant $r_\Omega$ depends only on $n$ and the uniform character of $\Omega$, and $r_\Omega=\infty$ when $\diam(\Omega)=\infty$.
\end{lemma}

In the case $n\geq2$, \rf{eq:Green-upperbound} says that
$$\omega^{x}(B)\approx \omega^{x_{B}}(B)\, r^{n-1}\, G(x,x_{B})\,\, \,\,\,\,\,\,\text{for all}\,\, x\in \Omega\backslash  (1+\tau)B.
$$
Recall that the inequality 
$$ \omega^{x}(B)\gtrsim \omega^{x_{B}}(B)\, r^{n-1}\, G(x,x_{B})\,\, \,\,\,\,\,\,\text{for all}\,\, x\in \Omega\backslash  B_0
$$
is already known to hold for arbitrary Greenian domains, as stated in 
\rf{eq:Green-lowerbound}. To prove the converse estimate we need to assume the domain to be uniform. 

Let us remark that in Lemma 3.6 of  Aikawa's work \cite{Ai1} it has been shown that
$$\omega^{x}(B)\lesssim  r^{n-1}\, G(x,x_{B})\,\, \,\,\,\,\,\,\text{for all}\,\, x\in \Omega\backslash  B_0.
$$
Clearly, the analogous inequality in \rf{eq:Green-upperbound} is sharper (at least in the case $n\geq 2$).
The essential tool for the proof of Lemma \ref{l:w<G} is the following
 boundary Harnack principle for uniform domains, also due to Aikawa \cite{Ai1}.

\begin{theorem} \label{t:BH}
 Let $\Omega\subset \R^{n+1}$ be a uniform domain. Then there are $A_1>1$ and $r_\Omega>0$ with the following property: Let $\xi \in \d\Omega$ and $0<r\leq r_\Omega$. Suppose $u,v$ are bounded positive harmonic functions on $\Omega \cap B(\xi,A_1r)$ vanishing quasi-everywhere on $\d\Omega \cap B(\xi,A_{0}r)$. Then
\begin{equation}\label{e:BH}
\frac{u(x)}{v(x)}\approx \frac{u(y)}{v(y)} \quad \mbox{ for all }x,y\in \Omega\cap B(\xi,r).
\end{equation}
The constant $r_\Omega$ depends only on $n$ the uniform character of $\Omega$, and $r_\Omega=\infty$ when $\diam(\Omega)=\infty$.
\end{theorem}

\begin{proof}[\bf Proof of Lemma \ref{l:w<G}]
We may assume that $0<\tau<1$.
Consider the annulus 
$$A_\xi := A(\xi,(1+\tau)r,2r),$$
where $\xi$ is the center of $B$.
We cover $\overline{A_\xi\cap\Omega}$ by a family of open balls $B_i$, $i\in I$, centered at $\xi_i\in\overline{A_\xi\cap\Omega}$,  all with radius equal to $c_2r$, where 
$c_2$ is some positive constant small enough so that $4A_1B_i\cap B=\varnothing$ for all $i\in I$.

From the discussion above and the Harnack chain condition, we infer that
\begin{equation}\label{eqdjk1}
G(y,x_B)\approx \rho(x_B)\quad \mbox{ if $|y-x_B|\approx r$ \,and \,$\dist(y,\partial\Omega)\gtrsim r$.}
\end{equation}
Also, by analogous arguments,
\begin{equation}\label{eqdjk2}
\omega^y(B)\approx \omega^{x_B}(B)\quad \mbox{ if $|y-x_B|\lesssim r$ \,and\, $\dist(y,\partial\Omega)\gtrsim r$.}
\end{equation}
Therefore, if $2B_i\cap \partial\Omega = \varnothing$, then
\begin{equation}\label{eqdjk3}
G(y,x_B)\approx \rho(x_B)\approx \rho(x_B)\,\frac{\omega^y(B)}{\omega^{x_B}(B)}\quad
\mbox{ for all 
$y\in  B_i\cap\Omega$.}
\end{equation}

Suppose now that $2B_i\cap \partial\Omega\neq\varnothing$, and take a ball $B_i'$ centered on $2B_i\cap\partial\Omega$
with radius $r(B_i')=4r(B_i)$, so that $2B_i\subset B_i'\subset 4B_i$, which,  in particular, implies that $A_1B_i'\cap
B=\varnothing$. For each ball $B_i'$, consider a corkscrew point $x_i\in B_i'$, that is, a point $x_i\in B_i'\cap
\Omega$ such that $\dist(x_i,\partial\Omega)\approx r(B_i')\approx r$, with the implicit constant depending on $\tau$, $A_1$
and other constants above. Then \rf{eqdjk1} and \rf{eqdjk2} hold for $y=x_i$, and thus also 
\begin{equation}\label{eqdjk4}
G(x_i,x_B)\approx \rho(x_B)\approx \rho(x_B)\,\frac{\omega^{x_i}(B)}{\omega^{x_B}(B)}.
\end{equation}
Since $A_1B_i'\cap
B=\varnothing$, and both $G(\cdot,x_B)$ and $w^{(\cdot)}(B)$ are bounded positive harmonic functions which vanish q.e. on $B_i'\cap\partial\Omega$,
by Aikawa's Theorem \ref{t:BH} and \rf{eqdjk4} we have
\begin{equation}\label{eqdjk5}
\frac{G(y,x_B)}{\omega^{y}(B)}\approx \frac{G(x_i,x_B)}{\omega^{x_i}(B)} 
\approx \frac{\rho(x_B)}{\omega^{x_B}(B)}
\quad \mbox{ for all 
$y\in  B_i'\cap\Omega$.}
\end{equation}

From \rf{eqdjk3} and \rf{eqdjk5} we infer that 
$$G(y,x_B)\approx  \rho(x_B)\,\frac{\omega^y(B)}{\omega^{x_B}(B)}\quad
\mbox{ for all 
$y\in  A_\xi\cap\Omega$.}$$
By the maximum principle, since both $G(\cdot,x_B)$ and   $\omega^{(\cdot)}(B)$ are bounded positive continuous harmonic functions in $\Omega\setminus B(\xi,(1+\tau)r)$ which vanish quasi-everywhere in $(\partial\Omega)\setminus B(\xi,(1+\tau)r)$, we deduce that
$$G(y,x_B)\approx  \rho(x_B)\,\frac{\omega^y(B)}{\omega^{x_B}(B)} \quad
\mbox{ for all 
$y\in  \Omega\setminus B(\xi,(1+\tau)r)$.}$$
\end{proof}

\begin{lemma}
\label{l:w/w}
Let $\Omega\subsetneq \R^{n+1}$, $n\geq1$, be a uniform domain and let $\tau>0$. Let $B,B'$ be balls centered on $\d\Omega$ so that $2B'\subseteq B$. Then for all $x\in \Omega\backslash (1+\tau)B$,
\begin{equation}
 \frac{\omega^{x}(B')}{\omega^{x}(B)} \approx_\tau  \frac{\omega^{x_B}(B')}{\omega^{x_B}(B)},
\label{e:w/w}
\end{equation}
where $x_B\in B\cap \Omega$ is a corkscrew point of $B$.
\end{lemma}

\begin{proof}
By the Harnack chain condition, we may assume that $x_B\in B\setminus (1+\tau)B'$.
By Lemma \ref{l:w<G}, we have that for all $x\in \Omega\backslash (1+\tau)B$, 
$$\omega^{x}(B)\approx \omega^{x_{B}}(B)\, \rho(x_B)^{-1}\,G(x,x_{B}),$$
$$\omega^{x}(B')\approx \omega^{x_{B'}}(B')\, \rho(x_{B'})^{-1}\,G(x,x_{B'}),$$
and
$$\omega^{x_B}(B')\approx \omega^{x_{B'}}(B')\, \rho(x_{B'})^{-1}\,G(x_{B'},x_{B}).$$
So
\begin{align*}  
 \frac{\omega^{x}(B')}{\omega^{x}(B)} & \approx 
 \frac{\omega^{x_{B'}}(B')\, \rho(x_{B'})^{-1}\,G(x,x_{B'})}{\omega^{x_{B}}(B) \,\rho(x_{B})^{-1}G(x,x_{B})}\approx 
\frac{\omega^{x_B}(B')}{\omega^{x_{B}}(B)}\,
\frac{
 G(x,x_{B'})}
 { \rho(x_{B})^{-1}\,G(x,x_{B})\,G(x_{B'},x_{B})}.
 \end{align*}
 
Thus the result will follow once we show 
\begin{equation}
 G(x,x_{B'})\approx \rho(x_{B})^{-1} G(x,x_{B})\,G(x_{B},x_{B'}).
 \label{GGG*}
 \end{equation}
 By the Harnack chain condition, it is immediate to check that this holds if $r(B)\approx r(B')$. Suppose that this is not the case, and assume then that $r(B')\leq \tau_0 r(B)$, for some $0<\tau_0\ll \tau\,A_1^{-1}$ to be fixed below. So if we consider an auxiliary ball $\wt B$ concentric
 with $B'$ of radius $r(\wt B)=\tau_0\,r(B)$, then we have
 $$B'\subset \wt B\subset 2A_1\wt B\subset (1+\tau)B.$$
 In particular, this tells us that $x\not \in 2A_1\wt B$, and thus 
 the function $u=G(x,\cdot )$ is harmonic and bounded in $A_1\wt B$. Further, by taking $\tau_0$ small enough, we also
 have $x_B\not\in A_1\wt B$, and then the function $v=G(x_B,\cdot )$ turns out to be harmonic in $A_1 \wt B$ too.
 Let $x_{\wt B}\in \wt B$ be a corkscrew point of $\wt B$.
 Note that by the Harnack chain condition,
 $$u(x_{\wt B}) =G(x,x_{\wt B}) \approx G(x,x_B),$$
and also
$$v(x_{\wt B}) = G(x_B,x_{\wt B}) \approx \rho(x_{B}).$$
 Since both functions $u$ and $v$
 vanish quasi-everywhere in $\partial\Omega$, by the boundary Harnack principle of Aikawa,
$$ \frac{G(x,x_{B'})}{G(x,x_B)}\approx
\frac{u(x_{B'})}{u(x_{\wt B})} \approx \frac{v(x_{B'})}{v(x_{\wt B})}
\approx G(x_B,x_{B'})\,\rho(x_{B})^{-1},$$
which proves \rf{GGG*} and thus the lemma. 
\end{proof}

\vv
\begin{rem}
Let $\Omega\subsetneq \R^{n+1}$, $n\geq1$, be a uniform domain and let $\tau>0$. Let $B$ be a ball centered on $\d\Omega$. By the preceding theorem, for all $x\in \Omega\backslash (2+\tau)B$,
$$\omega^x(2B) \approx_\tau \frac{\omega^{x_B}(2B)}{\omega^{x_B}(B)} \,\omega^x(B).$$
So if $\omega^{x_B}(B)\approx 1$,
then we deduce that 
$$\omega^x(2B) \approx_\tau \omega^x(B)$$
In particular, if $\Omega$ satisfies the so called {\em capacity density condition}, then $\omega^{x_B}(B)\approx 1$
for every ball $B$ centered on $\partial\Omega$ and thus $\omega^x$ is doubling. In this way, we recover a well known result of 
Aikawa and Hirata \cite{AH}.\footnote{In fact, in \cite{AH} it is shown that, under the capacity density condition, $\omega^x$ is doubling for the larger class of semi-uniform domains.}
\end{rem}

Now we are ready to prove Theorem \ref{teounif}, which we state again here for the reader's convenience.

\begin{theorem*}
Let $n\geq 1$ and let $\Omega$ be a uniform domain in $\R^{n+1}$. Let $B$ be a ball centered at $\partial\Omega$.
Let $p_1,p_2\in\Omega$ such that $\dist(p_i,B\cap \partial\Omega)\geq c_0^{-1}\,r(B)$ for $i=1,2$.
Then, for all $E\subset B\cap\d\om$,
$$\frac{\omega^{p_1}(E)}{\omega^{p_1}(B)}\approx \frac{\omega^{p_2}(E)}{\omega^{p_2}(B)},$$
with the implicit constant depending only on $c_0$ and the uniform behavior of of $\Omega$. 
\end{theorem*}

\begin{proof}
It is enough to show that for any  $p\in\Omega$ such that $\dist(p,B\cap \partial\Omega)\geq c_0^{-1}\,r(B)$,
\begin{equation}\label{eqaux11}
\frac{\omega^{p}(E)}{\omega^{p}(B)}\approx \frac{\omega^{x_B}(E)}{\omega^{x_B}(B)}.
\end{equation}
By Lemma \ref{l:w/w} and the Harnack chain condition it turns out that \rf{eqaux11} holds in the particular case when
$E$ equals some ball $B'$ such that $2B'\subset B$. Then, the comparability \rf{eqaux11} for arbitrary Borel sets $E$
follows by rather standard arguments. We show the details for the reader's
convenience.

By taking a sequence of open balls containing $B$ with radius converging to $r(B)$, it is easy to check
that we may assume the ball $B$ to be open. For an arbitrary $\ve>0$, consider an open set $U\subset B$ which contains $E$ and such that
$\omega^p(U\setminus E)\leq \ve$. By Vitali's covering theorem, we can find a family of disjoint balls $B_i$, $i\in I$, 
centered at $E$, with $2B_i\subset U$ for every $i\in I$, and such that $\bigcup_{i\in I}B_i$ covers $\omega^{x_B}$-almost
all $E$. So we have
\begin{align*}
\omega^{x_B}(E) & \leq\sum_i \omega^{x_B}(B_i) \lesssim \frac{\omega^{x_B}(B)}{\omega^{p}(B)} \sum_i \omega^{p}(B_i)\\
& \leq \frac{\omega^{x_B}(B)}{\omega^{p}(B)} \,\omega^p(U)
\leq \frac{\omega^{x_B}(B)}{\omega^{p}(B)} \,\bigl(\omega^p(E) + \ve\bigr).
\end{align*}
Letting $\ve\to0$,  we get
$$\frac{\omega^{p}(E)}{\omega^{p}(B)}\lesssim \frac{\omega^{x_B}(E)}{\omega^{x_B}(B)}.$$
The proof of the converse estimate is analogous.
\end{proof}

\vv

Finally we show how Theorem \ref{teo2} follows from Theorem \ref{teo1} in combination with the preceding result.

\begin{proof}[\bf Proof of Theorem \ref{teo2}]
The arguments are very standard but we give the details for the reader's convenience again. We assume that, for
 some point $p\in\Omega$, 
 there exist $\ve, \ve' \in (0,1)$ such that for every $(2,c_{db})$-doubling ball $B$ with $\diam(B)\leq 
 \diam(\Omega)$ centered at $\d \om$ the following holds: for any subset $E \subset B$,
 \begin{equation}\label{eq:Amu1*10}
\mbox{if}\quad\mu(E)\leq \ve\,\mu(B), \quad\text{ then }\quad \hm^{p}(E)\leq \ve'\,\hm^{p}(B).
\end{equation}

Fix $E$ and $B$ as above, so that  $\mu(E\cap B)\leq\ve\mu(B)$. Let $x_B$ be a corkscrew point for $\kappa B$. That is,
$x_B\in \kappa B\cap\Omega$ satisfies $\dist(x_B,\d\om)\approx r(B)$. 
By the assumption \rf{eq:Amu1*10},  $\hm^{p}(E)\leq \ve'\,\hm^{p}(B)$, and then by
Theorem \ref{teounif} we deduce that
$$(1-\ve)\leq\frac{\omega^p(E^c\cap B)}{\omega^p(B)} \leq C \,\frac{\omega^{x_B}(E^c\cap B)}{\omega^{x_B}(B)},$$
and thus
$$\omega^{x_B}(E\cap B)\leq (1-C^{-1}(1-\ve))\,\omega^{x_B}(B).$$
So the assumptions of 
 Theorem \ref{teo1} are satisfied and hence $\RR_\mu$ is bounded in $L^2(\mu)$.
\end{proof}

\vv

\section{The case when $\mu$ is AD-regular}\label{sec11}

Recall that if $\mu$ is an $n$-dimensional AD-regular measure in $\R^{n+1}$ and $\RR_\mu$ is bounded in $L^2(\mu)$, then
$\mu$ is uniformly $n$-rectifiable, by the Nazarov-Tolsa-Volberg theorem in \cite{NToV}.
So from Theorems \ref{teo1} we deduce:

\begin{coro}\label{coro1}
Let $n\geq 1$ and let $0<\kappa<1$ be some constant small enough depending only on $n$.
Let $\Omega$ be an open set in $\R^{n+1}$ and $\mu$ be an $n$-dimensional AD-regular measure supported on $\d \om$.
 Suppose that 
 there exist $\ve, \ve' \in (0,1)$ such that for every ball $B$ centered at $\supp\mu$ with $\diam(B)
 \leq\diam(\supp\mu)$
 there exists a point
 $x_B\in \kappa B\cap\Omega$ such that the
 following holds: for any subset $E \subset B$,
 \begin{equation}\label{eq*fgh}
\mbox{if}\quad\mu(E)\leq \ve\,\mu(B), \quad\text{ then }\quad \hm^{x_B}(E)\leq \ve'\,\hm^{x_B}(B).
\end{equation}
Then $\mu$ is uniformly $n$-rectifiable.
\end{coro}
\vv

Given a Radon measure $\sigma$,
we write $\sigma\in A_\infty(\mu)$ if there exist $\ve, \ve' \in (0,1)$ such that for every ball $B$ centered at $\supp\mu$ 
 with $\diam(B)
 \leq\diam(\supp\mu)$ the
  following holds: for any subset $E \subset B$,
$$
\mbox{if}\quad\mu(E)\leq \ve\,\mu(B), \quad\text{ then }\quad \sigma(E)\leq \ve'\,\sigma(B).
$$
From Theorem \ref{teo2} we obtain the following:

\begin{coro}\label{coro2}
Let $n\geq 1$, $\Omega$ be a bounded uniform domain in $\R^{n+1}$ and $\mu$ be an $n$-dimensional AD-regular measure supported on $\d \om$.
 Let $p\in\Omega$ and suppose that $\omega^p\in A_\infty(\mu)$.
Then $\mu$ is uniformly $n$-rectifiable.
\end{coro}

\vv

It is worth comparing Corollary \ref{coro1} with the main result of the work \cite{HM4} of Hofmann and Martell, which reads as follows:

\begin{theoremA}[\cite{HM4}]
Let $\Omega$ be an open set in $\R^{n+1}$, with $n\geq 2$, whose
boundary is $n$-dimensional AD-regular. Suppose that there exists some constant $C_6\geq 1$ 
and an exponent $p>1$ such that, for every ball $B=B(x,r)$ with $x\in\partial\Omega$, $0<r\leq\diam(\Omega)$,
there exists $x_B\in\Omega\cap B(x,C_6r)$ with $\dist(x_B,\partial\Omega)\geq C_6^{-1}r$ satisfying
\begin{enumerate}[(a)]
\item { Bourgain's estimate:} $\omega^{x_B}(B) \geq C_6^{-1}$.
\item { Scale-invariant higher integrability:} $\omega\ll\HH^n|_{\partial\Omega}$ in $C_7B$ and
 \begin{equation}\label{eq*fgh2}
\int_{C_7B\cap\partial\Omega} \left(\frac{d\omega^{x_B}}{d\HH^n}(y)\right)^p\,d\sigma(y)\leq C_6\,\HH^n(C_7B
\cap\partial\Omega)^{1-p}.
\end{equation}
where $C_7$ is a sufficiently large constant
depending only on $n$ and the AD-regularity constant of $\partial\Omega$.
\end{enumerate}
Then $\partial\Omega$ is uniformly $n$-rectifiable.
\end{theoremA}

\vv
Observe that the assumption (a) in the last theorem is guarantied by Lemma \ref{lembourgain} if we assume that
$x_B\in\delta_0B = \kappa2B$, taking into account the AD-regularity of $\partial\Omega$. So if moreover we assume $C_7\geq2$, then 
 from the condition \rf{eq*fgh2} in Theorem A, for any set $E\subset 2B$, writing $\sigma:=\HH^n|_{\partial\Omega}$, we get
\begin{align*}
\omega^{x_B}(E) & = \int_{E} \frac{d\omega^{x_B}}{d\sigma}(y)\,d\sigma(y)\\
& \leq \sigma(E)^{1/p'}\left(\int_{2B} \left(\frac{d\omega^{x_B}}{d\sigma}(y)\right)^p\,d\sigma(y)\right)^{1/p}
\leq C_6\,\sigma(E)^{1/p'}\sigma(C_7B)^{-1/p'},
\end{align*}
Using the fact that $\sigma$ is doubling and the condition assumption (a) in the Theorem A we obtain
$$\omega^{x_B}(E) \leq C\,\left(\frac{\sigma(E)}{\sigma(2B)}\right)^{1/p'}\leq C'\,\left(\frac{\sigma(E)}{\sigma(2B)}\right)^{1/p'}
\,\omega^{x_B}(2B).$$
This implies that the condition \rf{eq*fgh} in Corollary \ref{coro1}, with $\mu =\sigma$, is satisfied by $2B$.
Thus the corollary ensures that $\partial \Omega$ is uniformly rectifiable. To summarize, Theorem A is a consequence
of Corollary \ref{coro1} if we we suppose that $C_7\geq2$ and we replace the assumption (a) in the theorem by the (quite natural) assumption that
$x_B\in \delta_0B$.

On the other hand, note that the support of $\mu$ in Corollary \ref{coro1} may be a subset
strictly smaller than $\partial \Omega$ and so this can be considered as a local result.
Observe also that in the corollary we allow $n=1$ and we do not ask the pole $x_B$ for harmonic measure to satisfy $\dist(x_B,\partial\Omega)\gtrsim r(B)$, unlike in Theorem A. However, this latter improvement 
is only apparent because, as Steve Hofmann explained to us \cite{Hofmann-private}, it turns out that the assumption \rf{eq*fgh} implies that
$\dist(x_B,\partial\Omega)\gtrsim r(B)$ when $\mu$ is AD-regular.

\vv

In connection with harmonic measure in uniform domains, Hofmann, Martell and Uriarte-Tuero \cite{HMU} proved the following:

\begin{theoremB}[\cite{HMU}]\label{teob}
Let $n\geq 2$, $\Omega$ be a bounded uniform domain in $\R^{n+1}$ 
whose
boundary is $n$-dimensional AD-regular.
 Let $p\in\Omega$ and suppose that $\omega^p\in A_\infty(\HH^n|_{\partial \Omega})$.
Then  $\partial\Omega$ is uniformly $n$-rectifiable.
\end{theoremB}

Corollary \ref{coro2}, which also applies to the case $n=1$, can be considered as a local version of this result, because 
the support of $\mu$ is allowed to be
strictly smaller than $\partial \Omega$, analogously to Corollary \ref{coro1}.

\vvv


\begin{thebibliography}{KKKMM}

  
\bibitem[AHM3TV]{AHM3TV} J. Azzam, S. Hofmann, J.M. Martell, S. Mayboroda, M. Mourgoglou, X. Tolsa, and A. Volberg. {\em Rectifiability of harmonic measure.} Preprint (2015).  arXiv:1509.06294

\bibitem[AHMNT]{AHMNT} J. Azzam, Steve Hofmann, J.M. Martell and K. Nystr\"om and T. Toro,
{\em A new characterization of chord-arc domains.} To appear in J. Eur. Math. Soc.	arXiv:1406.2743.


\bibitem[Ai1]{Ai1} H. Aikawa. {\em Boundary Harnack principle and Martin boundary for a uniform domain.} J. Math. Soc. Japan.
Vol. 53, no. 1 (2001), 119--145.

\bibitem[Ai2]{Ai2} H. Aikawa. {\em Equivalence between the boundary Harnack principle and the Carleson estimate.} Math. Scand. 103 (2008), 61--76.  

\bibitem[AiH]{AH} H. Aikawa and K. Hirata. {\em Doubling conditions for harmonic measure in John domains.} Ann. Inst. Fourier 58 (2008), no. 2, 429--446.

\bibitem[AMT]{AMT} J. Azzam, M. Mourgoglou, and X. Tolsa. {\em Rectifiability of harmonic measure in domains with porous boundaries}, arXiv:1505.06088. 

\bibitem[AT]{AT} J. Azzam and X. Tolsa. {\em Characterization of 
$n$-rectifiability in terms of Jones' square function: Part II.} Preprint (2014). To appear in Geom. Funct. Anal. (GAFA).


  
\bibitem[BH]{BH} S. Bortz and S. Hofmann, {\em Harmonic measure and approximation of uniformly
rectifiable sets.} Preprint (2015).  arXiv:1505.01503.
  
  
  
\bibitem[Bo]{Bo}
J.~Bourgain, \emph{On the {H}ausdorff dimension of harmonic measure in higher
  dimension}, Invent. Math. \textbf{87} (1987), no.~3, 477--483. 
  
  \bibitem[Da]{Da}
B.~E.~J. Dahlberg, \emph{Estimates of harmonic measure}, Arch. Rational Mech.
  Anal. \textbf{65} (1977), no.~3, 275--288. 

\bibitem[DJ]{DJ}
G.~David and D.~Jerison, \emph{Lipschitz approximation to hypersurfaces,  harmonic measure, and singular integrals}, Indiana Univ. Math. J. \textbf{39}  (1990), no.~3, 831--845. 

\bibitem[DM]{David-Mattila} G. David and P. Mattila. {\em Removable sets for Lipschitz
harmonic functions in the plane.} Rev. Mat. Iberoamericana 16(1) (2000),
137--215.

\bibitem[DS]{DS} G. David and S. Semmes, \emph{Analysis of and on uniformly
rectifiable sets}, Mathematical Surveys and Monographs, 38. American
Mathematical Society, Providence, RI, (1993).


\bibitem[GM]{GM}
J.~B. Garnett and D.~E. Marshall, \emph{Harmonic measure}, New Mathematical
  Monographs, vol.~2, Cambridge University Press, Cambridge, 2008. 



\bibitem[Hel]{Hel} L.L. Helms. {\em Potential theory, 2nd Ed.} Springer, London 2014.

\bibitem[HM1]{HM1} S. Hofmann and J.M. Martell, {\em Uniform Rectifiability and Harmonic Measure I: Uniform
rectifiability implies Poisson kernels in $L^p$},  Ann. Sci. \'Ecole Norm. Sup. 47 (2014), no. 3, 577--654.

\bibitem[HM2]{HM4} S. Hofmann, J.M. Martell. {\em Uniform rectifiability and harmonic measure, IV: Ahlfors regularity plus Poisson kernels in $L^p$ impies uniform rectifiability.}  
 arXiv:1505.06499.



\bibitem[HMMTV]{HMMTV} S. Hofmann, J.M. Martell, S. Mayboroda, X. Tolsa and A. Volberg. {\em
Absolute continuity between the surface measure and harmonic measure implies rectifiability}. To appear in Geom. Fun. Anal. 
 arXiv:1507.04409.

\bibitem[HMU]{HMU} S. Hofmann, J.M. Martell, I. Uriarte-Tuero. {\em Uniform rectifiability and harmonic measure, {II}: {P}oisson kernels in {$L^p$} imply uniform rectifiability.} Duke Math. J. (2014) no. 8, p. 1601-1654.

\bibitem[Ho]{Hofmann-private} S. Hofmann. {\em Non-degeneracy of harmonic measure plus ADR implies
corkscrew.} Private communication (2015).

\bibitem[JK]{JK}
D.~S. Jerison and C.~E. Kenig, \emph{Boundary behavior of harmonic functions in
  nontangentially accessible domains}, Adv. in Math. \textbf{46} (1982), no.~1,
  80--147. 



\bibitem[NToV1]{NToV} F. Nazarov, X. Tolsa and A. Volberg, \emph{On the uniform
  rectifiability of {AD}-regular measures with bounded {R}iesz transform
  operator: the case of codimension 1}, Acta Math. \textbf{213} (2014), no.~2,
  237--321. 


\bibitem[NToV2]{NToV-pubmat} F. Nazarov, X. Tolsa and A. Volberg. {\em The Riesz transform Lipschitz harmonic functions.}
Publ. Mat. 58 (2014), 517--532.

\bibitem[NTrV]{NTV} {F. Nazarov, S. Treil and A. Volberg.}
{\em The Tb-theorem on non-homogeneous
spaces that proves a conjecture of Vitushkin.} CRM preprint No. 519 (2002), pp. 1--84.



\bibitem[To1]{Tolsa-bilip} X. Tolsa, {\em Bilipschitz maps, analytic
capacity, and the Cauchy integral}, Ann. of Math. 162:3 (2005), 1241--1302.

\bibitem[To2]{Tolsa-llibre}
X.~Tolsa.
{\em Analytic capacity, the {C}auchy transform, and non-homogeneous
  {C}alder\'on-{Z}ygmund theory}, volume 307 of {\em Progress in Mathematics}.
 Birkh\"auser Verlag, Basel, 2014.

\bibitem[To3]{Tolsa-memo} X. Tolsa. {\em Rectifiable measures, square functions involving densities, and the Cauchy transform.}
To appear in Mem. Amer. Math. Soc.

\bibitem[Vo]{Volberg} A.\ Volberg, {\em Calder\'on-Zygmund capacities and operators on nonhomogeneous spaces.}
CBMS Regional Conf. Ser. in Math. 100, Amer. Math. Soc., Providence, 2003.




\end{thebibliography}
\end{document}